\documentclass[11pt]{article}

\usepackage{amsfonts,amsthm,amsmath,latexsym,amssymb}

\usepackage[dvips]{graphics}

\usepackage{epsfig}
\usepackage{graphicx}
\usepackage{color}
\usepackage{hyperref}

\usepackage{epsf}
\usepackage{psfrag}
\newdimen\epsfxsize
\newdimen\epsfysize

\input epsf.tex
\newdimen\epsfxsize
\newdimen\epsfysize

\newcommand{\be}{\begin{equation}}
\newcommand{\ee}{\end{equation}}
\newcommand{\bes}{\begin{equation*}}
\newcommand{\ees}{\end{equation*}}
\newcommand{\s}{\sigma}

\renewcommand{\l}{\lambda}
\renewcommand{\t}{\tau}
\newcommand{\g}{\gamma}
\newcommand{\G}{\Gamma}
\newcommand{\C}{\mathbb C}
\renewcommand{\H}{\mathbb H}
\newcommand{\DD}{\mathbb D}
\newcommand{\R}{\mathbb R}

\newcommand{\dist}{\mathrm{dist}}
\newcommand{\diam}{\mathrm{diam\,}}

\renewcommand{\Im}{\operatorname{Im}}

\newcommand{\hcap}{\operatorname{hcap}}
\newcommand{\Ca}{\text{Cap}}

\newtheorem{thm}{Theorem}[section]
\newtheorem{prop}[thm]{Proposition}
\newtheorem{remark}[thm]{Remark}
\newtheorem{defn}[thm]{Definition}
\newtheorem{cor}[thm]{Corollary}

\newtheorem{lemma}[thm]{Lemma}

\makeatletter

\@addtoreset{equation}{section} \makeatother

\def\eps{\varepsilon}
\def\1{{\bf 1}}

\begin{document}

\begin{doublespace}

\title{\bf Collisions and Spirals of Loewner Traces}
\bigskip
\author{{\bf Joan Lind, } {\bf Donald E. Marshall}\footnote{Research
supported in part by NSF Grant DMS-0602509} ~{\bf and}  {\bf Steffen
Rohde}\footnote{Research supported in part by NSF Grants DMS-0501726 and DMS-0800968.}
}

\maketitle

\abstract{We analyze Loewner traces driven by functions asymptotic to
$\kappa\sqrt{1-t}$.
We prove a stability result when $\kappa\neq4$ and show that $\kappa=4$ can lead to 
non locally connected hulls. As a consequence, we obtain a driving
term $\lambda(t)$ so that 
the hulls driven by $\kappa \lambda(t)$ are generated by a continuous curve 
for all $\kappa>0$ with $\kappa\neq 4$ but not when $\kappa=4,$
so that the space of driving terms with continuous traces is not convex.
As a byproduct, we obtain an explicit construction of
the traces driven by $\kappa\sqrt{1-t}$ and a conceptual proof of the corresponding results of 
Kager, Nienhuis and Kadanoff.}

\tableofcontents

\bigskip

\section{Introduction and Results}\label{s0}

 \bigskip

Let $\lambda(t)$ be continuous and real valued and let
$g_t:\H\setminus K_t\to\H$ be the solution to the Loewner equation
\be\label{ODE}
\frac{d}{dt} g_t(z) = \frac{2}{g_t(z)-\l(t)}\quad , \quad g_0(z)=z\in
\H,
\ee
where $\H$ is the upper half-plane.
It was shown in \cite{MR1} and \cite{L} that if $\l$ is H\"older continuous
with exponent $1/2$ and if $||\l||_{1/2}<4$, then 
there is a simple curve $\gamma$ with
$\gamma[0,t]=K_t$ and $\gamma\setminus\gamma(0)\subset \H.$ The norm $4$ is sharp as the
examples $\l(t)=\kappa\sqrt{1-t}$ show: Indeed, by \cite{KNK}, $\gamma$ touches back on the real line
if $\kappa\geq4$ 
(hence the driving term $\l(t)=\kappa$ for $0\leq t\leq t_0$ and
$\l(t)=\kappa\sqrt{t_0+1-t}$
for $t_0\leq t\leq t_0+1$ has a self-intersection in $\H$ for $t_0$
sufficiently large). It was also shown in 
\cite{MR1} that there
is a $\l$ with $||\l||_{1/2}<\infty$ such that $K_1$ spirals infinitely often
around some disc, and hence
is not locally connected. The starting point of this paper is the observation that from the conformal
mapping point of view, the zero angle cusp at the tangential self-intersection for  
$\l(t)=4\sqrt{1-t}$ is very similar to the infinitely spiraling prime end, and that this is reflected
in the driving terms: 

\begin{thm}\label{t:spiral} If $\gamma$ is a sufficiently smooth
infinite spiral of half-plane capacity $T$, or if $\gamma$ has a tangential self-intersection, then 
its driving term $\l$ satisfies
$$ \lim_{t \rightarrow T} \frac{|\lambda(T)-\lambda(t)|}{\sqrt{T-t}}=4.$$
\end{thm}

\begin{figure}[h]
\centering
\centerline{\includegraphics[height=1.65in,angle=180]{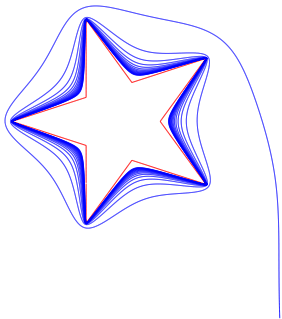}}
\caption{An infinite spiral converging towards a star.
}\label{star}
\end{figure}

\noindent See Sections \ref{s:definitions} and \ref{s:spiral} for the definitions and precise statements.
In Section \ref{s:spiral} we show that for every compact connected set 
$A\subset\H$ with connected complement, there is a
sufficiently smooth infinite spiral winding infinitely often around $A$ with 
limit set $\partial A$; see Figure \ref{star}. 
The following natural question has been asked by Omer Angel: If the hull of $\l$ is generated
by a continuous curve $\gamma$ and if $r<1$, is it true that the hull of $r\l$ is generated by a continuous
curve, too? In other words, is the space of driving terms of continuous curves starlike?
We answer this question in the negative by proving

\begin{thm}\label{t:continuity} If $\gamma$ is a sufficiently smooth infinite
spiral of half-plane capacity $T$, and if 
$\l$ is its driving term, then the trace of $r\l$ is continuous on the
closed interval $[0,T]$ for all $r\neq\pm1.$
\end{thm}

\noindent
The main work is in proving a form of stability of the (nontangential)
self-intersection of $\l(t)=\kappa\sqrt{1-t}$
for $\kappa>4:$

\begin{thm}\label{t:collision} If $\l:[0,T]\to\R$ is sufficiently regular on $[0,T)$ and
if 
$$ \lim_{t \rightarrow T}
\frac{|\lambda(T)-\lambda(t)|}{\sqrt{T-t}}=\kappa>4,$$
then
$$\gamma(T) = \lim_{t\to T} \gamma(t)$$
exists, is real and $\gamma$ intersects $\R$ in the same angle as the
trace for $\kappa\sqrt{1-t}$.
\end{thm}

\vskip.2in\noindent
See Section \ref{s:collisions} for the statement of the necessary regularity.
A similar result is true for $\kappa<4$, see Theorem \ref{t:sp} in Section
\ref{s:collisions}. By Theorems \ref{t:collision} and \ref{t:sp},
the proof of Theorem \ref{t:continuity} is reduced to proving sufficient regularity
of the driving term of sufficiently smooth spirals. This is carried out in 
Proposition \ref{p:sufficient}.

As mentioned above, the solutions to the Loewner equation driven by
$\l(t)=\kappa\sqrt{1-t}$ were first computed
in \cite{KNK}. Their solutions are somewhat implicit and their analysis of the behaviour at the tip involved 
a little work. Our proof of Theorem \ref{t:collision} is based on the fact that 
the traces of $\l(t)=\kappa\sqrt{1-t}$ are fixed points of a certain renormalization 
operator, and that they take an extremely simple shape
(they are straight lines and logarithmic spirals)
after an appropriate change of coordinates. 
We therefore obtain an explicit ``geometric construction'' of the trace, which might be
of independent interest. See Sections \ref{s:scaling} and
\ref{s:self-similar}. 
We also need conditions and results about closeness of 
traces assuming closeness of driving terms, and vice versa.
These are stated and proved in Sections \ref{s:uniform} and \ref{s:nearby}.

\vskip.2in\noindent
{\bf Acknowledgement:} We would like to thank Byung-Geun Oh for our conversations about Theorem \ref{t:spiral}.

\section{Basics}\label{s:basics}

\subsection{Definitions and first properties}\label{s:definitions}
In this section, we will fix some notation and terminology,
as well as collect some standard properties. The expert can safely skip
this section.

A {\it hull} is a bounded set
$K\subset\H$ is such that $\H\setminus K$ is connected and simply connected.
If $g_K$ is a conformal map of
$\H\setminus K$ onto $\H$ such that $|g_K(z)|\to \infty$ as $z
\to \infty$, let $\widetilde
K= \overline{K\cup K^\R}\cup_j I_j$, where $K^\R$ is the
reflection of $K$ about $\R$  and $\{I_j\}$ are the bounded intervals
in $\R\setminus \overline{K\cup K^R}$. Then by the Schwarz reflection 
principle, $g_K$
extends to be a conformal map of $\C^*\setminus {\widetilde K}$ onto
$\C^*\setminus I$ where $\C^*$ is the extended plane and 
and $I$ is an interval contained in $\R$.
Composing with a linear map $az+b$, $a>0$, $b\in \R$, we may suppose
that $g_K$ has the hydrodynamic normalization
\be\label{normalization}
g_K(z)=z+\frac{2d}{z}+ {\text O}(\frac{1}{z^2})
\ee
near $\infty$. If $f(z)\equiv g_K^{-1}(z)=z-2d/z+\dots$ 
is continuous on $\overline\H$ then
\be\label{cauchy}
f(z)-z=\int_I\frac{\Im f(x)}{x-z}\frac{dx}{\pi},
\ee
by the Cauchy integral formula or by the Poisson
integral formula in $\H$ applied to the bounded harmonic function
$\Im(f(z)-z)$.
Note that (\ref{cauchy}) implies that
\be\label{2d}
2d=\lim_{z\to\infty} -z(f(z)-z)=\frac{1}{\pi}\int_I\Im f(x) dx > 0,
\ee
unless $f(z)\equiv z$.
The coefficient $d$ is called the half-plane capacity of $K$ and is
denoted by $d=\hcap(K)$.
It is easy to see
that  $\hcap$ is strictly increasing.

If $\l:[0,T]\to\R$ is continuous and $z\in\H$
then there are two cases for the solution $g_t(z)$ to the initial value
problem (Loewner equation)
\be
\frac{d}{dt} g_t(z) = \frac{2}{g_t(z)-\l(t)}\quad , \quad g_0(z)=z.
\ee 
Either there is a time $T_z\leq T$ such that $\liminf_{t\to T_z} |g_t(z)-\l(t)|=0$
(in this case it is not hard to show that  $\lim_{t\to T_z} |g_t(z)-\l(t)|=0$),
or $\inf_{t\in[0,T]}|g_t(z)-\l(t)|>0$. Set $T_z=\infty$ in the latter case.
If
$$K_t = \{z\in\H: T_z\leq t\},$$
then $\H\setminus K_t$ is simply connected, and $g_t:\H\setminus K_t\to\H$ is the (unique) conformal
map with $g_t(z) = z + 2t/z + O(1/z^2)$ near infinity. Thus each $K_t$ is a
hull and $\hcap(K_t)=t.$
We say that the hulls $K_t$ are {\it driven by $\lambda$} and that
$\lambda$ is the {\it driving term for $K_t$}.
We also say that $K_t$ is {\it generated by a curve $\gamma$} 
if there is a continuous function $\gamma:[0,T]\to\overline\H$
such that for each $t\in[0,T],$ the domain $\H\setminus K_t$ is 
the unbounded component of 
$\H\setminus\gamma[0,t].$ The curve $\gamma$ is called the {\it
trace} and we also say that $g$ and $\gamma$ are driven by $\l$ and
use the notation $g^\l$ and $\gamma^\l$ if necessary. 
It is known (see \cite{MR2}) that 
the hulls driven by a sufficiently regular $\l$ are simple (Jordan)
curves, but that there are continuous $\l$ whose hulls are not locally 
connected and hence not generated by a curve.

Consider a sequence of continuously growing hulls $K_t$ with
$K_0=\emptyset$
(see \cite{La} for a precise definition). Re-parametrizing $K_t$ if necessary, 
we may assume that $\hcap(K_t)=t$.
Then the hydrodynamically normalized conformal maps
$g_t\equiv g_{K_t}:\H\setminus K_t \to \H$ satisfy the Loewner equation for some continuous
function $\lambda(t)$ and $K_t$ are the hulls driven by $\l$.
If $g_t^{-1}$ has a continuous extension to $\l(t)$ 
then $g_t^{-1}(\l(t))$ is well-defined. If furthermore
$\g(t)=g_t^{-1}(\l(t))$ is a continuous curve, 
then $K_t=\text{fill}(\gamma[0,t]),$
where $\text{fill}(A)$ denotes the union of $A$ and the bounded components of
$\H\setminus A$, that is 
the complement of the unbounded component of $\H\setminus A$. 

The standard example is provided by 
a continuous curve $\gamma\in\overline\H$, beginning in $\R$ and 
without self-crossings but possibly self-touching, and 
$K_t=\text{fill}(\gamma[0,t])$,
In this case, $g_t(\gamma(t)) = \l(t).$
Notice that in general,
the trace $\gamma[0,t]$ is only a subset of the hull $K_t$, 
unless $\gamma$ is a simple curve.
For example, the hulls $K_t$ on the middle left of Figure \ref{gt} are equal to the 
trace $\gamma[0,t]$ for all $t<1$ (the $\kappa$ in the figure is a parameter), but $K_1$ equals 
$\gamma[0,1]$ together with the whole region enclosed by $\g.$

A crucial property is {\it scaling}: From
$$ g_{r K}( z) = r g_K(\frac{z}{r})$$
it follows that 
$$\hcap(r K) = r^2 \hcap(K),$$
and that scaled hulls $r K_t$ are driven by $\frac1r \l(r^2 t)$,
if $K$ is driven by $\l.$ Since the function $\l(t)=\kappa\sqrt{t}$ is invariant
under the scaling $\l \mapsto \frac1r \l(r^2 t)$, it follows that its hulls
are invariant under the geometric scaling $K\mapsto r K$. Notice that this
would immediately imply that the hulls are rays $K_r = a r^2 e^{i\theta}$ for 
some $a(K)>0$, if we assume that $K_r$ is generated by a simple curve.
This of course also can be done by a direct computation.

Other crucial simple properties are the behaviour under {\it
translation}
(because $g_{K+x}(z) = g_K(z-x)+x$,
the driving term of $\g+x$ is $\l+x$), under {\it concatenation}
(if $K_1$ and $K_2$ are hulls driven by $\l_1:[0,t_1]\to\R$ and $\l_2:[0,t_2]\to\R$
and if $\l_1(t_1)=\l_2(0),$ then $K_1*K_2=K_1\cup g_{K_1}^{-1}(K_2)$ is driven by
$\l(t)=\l_1(t) 1_{[0,t_1]} + \l_2(t-t_1) 1_{(t_1,t_1+t_2]}$), and
under {\it reflection} (if $R_I$ denotes reflection in the imaginary axis,
then $g_{R_I(K)}= R_I\circ g_K\circ R_I$ so that $R_I(K)$ is driven by $-\l$).
We will often use the following version of the above concatenation:
If $\g[0,t]$ is driven by $\l$, then $g_T(\g[T,t])$ is driven by
$\tau\mapsto\l(T+\tau)$, for $0\le \tau\le t-T$.

\subsection{Renormalization on [0,1)}\label{s:scaling}
Let $\l$ be continuous on $[0,1)$ and assume for ease of notation
that the associated hulls $K_t$ are generated by a curve $\g(t)$,
$0\le t < 1$. 
In order to understand the trace $\g$ (more generally the hulls $K$) near $t=1$,
we want to ``pull down'' the initial part $\gamma[0,T]$ of the curve by applying $g_T$,
and then rescale the result so as to have half-plane capacity 1 again.
For fixed $T\in[0,1)$, the curve $\widetilde{\g}_T = g_T(\g[T,1))$ that is
parametrized by
\begin{equation}\label{renormgt}
\widetilde{\g}_T(t) = g_T(\g(T+t)), \quad 0\leq t < 1-T
\end{equation}
is driven by 
\begin{equation}\label{renormlt}
\widetilde{\l}_T(t) = \l(T+t), \quad 0\leq t < 1-T.
\end{equation}
Since $\widetilde{\g}_T$ has capacity $1-T$, the scaled copy of
$\widetilde{\g}_T$
\begin{equation}\label{renormg}
\g_T(t) \equiv \widetilde{\g}_T\bigl(t (1-T)\bigr)/\sqrt{1-T}, \quad
0\leq t < 1
\end{equation}
has half-plane capacity 1.
By Section \ref{s:definitions}, $\g_T$ is the Loewner trace of
\begin{equation}\label{renorml}
\l_T(t) = \l\bigl(T+ t (1-T)\bigr)/\sqrt{1-T}, \quad 0\leq t < 1.
\end{equation}

\subsection{A time change}\label{s:timechange}

To facilitate our analysis of curves with driving term asymptotic to
$\kappa \sqrt{1-t}$,
we would like to reparametrize $\g$ in a way that is well adapted to the
renormalization operation \eqref{renormg}. 
Let $\g$ be a curve paramatrized by half-plane capacity $t\in[0,1].$
If $\g(T)$ and $\g(t)$ are consecutive points 
($0\leq T<t\leq 1$), then the renormalization of the arc between $\g(T)$ and
$\g(t)$ has half-plane capacity 
$(t-T)/(1-T)$. In other words
\bes
\frac{g_T(\gamma(t))}{\sqrt{1-T}}=
\gamma_T\Bigl(\frac{t-T}{1-T}\Bigl).
\ees
A parametrization $s(t)$ leaves ``time-differences''
invariant under renormalization provided
$$s(t)-s(T)= s\Bigl(\frac{t-T}{1-T}\Bigl)-s(0).$$
Dividing by $t-T$, passing to the limit $T\to t$ and integrating
(after setting $s(0)=0$ and $s'(0)=1$), 
we therefore define
\begin{equation}\label{e:timechange}
s=s(t) = \log{\frac1{1-t}},\quad \text{or} \quad t = 1-e^{-s},
\end{equation}
where $0\le s < \infty$. Set
\begin{equation}\label{e:defgs}
G_s(z) =  \frac{g_t(z)}{\sqrt{1-t}},\quad 
F_s = G_s^{-1},\quad
\s(s) = \frac{\l(t)}{\sqrt{1-t}},\quad \text{ and }\quad
\Gamma(s)=\gamma(t)
\end{equation}
so that
\bes 
G_s(\G(s)) = \s(s)\quad {\text and} \quad 
F_s(\s(s))=\G(s)\ees
We will say that $\G$, $G$ and $F$ are {\it driven by} $\s$ and write
$\G^{\s},G^{\s}$ and $F^{\s}$ if neccessary. 
By (\ref{e:defgs}) and (\ref{ODE})
\be\label{LDE-G}
\dot G_s\equiv\frac{\partial}{\partial s} G_s=
 \frac2{G_s-\s(s)} + \frac{G_s}{2}
\ee
for all $z\in \H\setminus \G[0,s],$ and
\begin{equation}\label{LDE-F}
\frac{\dot F_s}{F_s'} = \frac{2}{\s-z} - \frac{z}{2}
\end{equation}
for all $z\in\H.$ 
This change of variables was used in
\cite{KNK} when $\lambda(t)=\kappa\sqrt{1-t}$, in which case
$\sigma(s)\equiv \kappa$.

\bigskip\noindent
{\bf Convention:} Throughout the remainder of the paper, the symbol $s$ will refer to the 
``time change'' defined by \eqref{e:timechange}, whereas $t$ 
will always stand for the parametrization
by half-plane capacity.

\bigskip\noindent
We will now express the scaling relation \eqref{renorml} in terms of $s$ and establish
the semigroup property of $G_s$. The simple form may be the main advantage of the time change.
Denote the shift of a function $\s$ on $[0,\infty)$ by $\s_u,$
$$\s_u (s) = \s(u+s) \quad {\text for}\quad s\geq0.$$
To simplify the notation, we set $\Gamma_{u,v}=G_u(\Gamma[u,v])$ and
$\Gamma_u=\Gamma_{u,\infty}$.
Then $\Gamma_{u,v}$ is the ``pull-back'' of the portion of 
$\Gamma$ between ``s-times'' $u$ and $v$, with initial point 
$\sigma(u)=\l(t(u))\in \R$.

\begin{lemma}\label{l:semigroup} If the curve $\G$ is driven by $\s,$ then
the curve $\Gamma_u$ is driven by
$\s_u.$ Moreover,
$$G_{u+s}^{\s} = G_s^{\s_u} \circ G_u^{\s} \ \ .$$
\end{lemma}

\begin{proof}
Fix $u$ and set $\tau=1-e^{-u}.$ By \eqref{renormgt}, 
\eqref{renormg}, and \eqref{e:defgs},
$G_u^{\s}(\G[u,\infty])$ is the curve $\g_\tau$. By \eqref{renorml}
$\g_\tau$ 
is driven by $\l_\tau.$ Writing $t = 1-e^{-s},$ we have
$$\frac{\l_\tau(t)}{\sqrt{1-t}} =
\frac{\l(1-e^{-u}+(1-e^{-s})e^{-u}))}{e^{-u/2}e^{-s/2}} = 
\frac{\l(1- e^{-(u+s)})}{e^{-(u+s)/2}}=\s_u(s)$$
and hence $G_u^{\s}(\G[u,\infty])$ is driven by $\s_u$.
The semigroup property follows because both maps $G_{u+s}^{\s}$ and
$G_s^{\s_u} \circ G_u^{\s}$ are
normalized conformal maps of the same domains, hence identical.
\end{proof}

\subsection{Table of Notation and Terminology}

\bigskip

\begin{tabular}{c|c}
\text{Notation} & \text{Brief definition} \\ \hline
hull & bounded subset of $\H$ with simply connected complement in
$\H$\\
$g_K$ & normalized conformal map $\H\setminus K$ onto $\H$\\
$\lambda(t)$ & Loewner driving term \\
$K_t$ & Loewner hull \\
$\gamma$ & trace\\
$g_t\equiv g_{K_t}$ & Loewner map from $\mathbb{H} \setminus K_t$ to  $\mathbb{H}$ \\
$f_t$ & $g_t^{-1}$  \\
$\widetilde{\gamma}_T$ & $g_T(\gamma[T, 1])$\\ 
$\gamma_T$ & $ g_T(\gamma[T,1])/\sqrt{1-T}$ \\
$\lambda_T(t)$ &  $ \lambda(T+t(1-T))/\sqrt{1-T}$ \\
$s$ & $-\ln(1-t)$, and so  $t=t(s)=1-e^{-s}$ \\
$\Gamma(s)$ & $\gamma(t(s))$\\
$G_s$ & $e^{s/2}g_{1-e^{-s}}(z) = g_{t(s)}(z)/\sqrt{1-t(s)}$ \\
$F_s$ &  $G_s^{-1}$ \\
$\sigma(s)$ & $e^{s/2} \lambda(1-e^{-s})=\lambda(t(s))/\sqrt{1-t(s)}$ \\
$g_t^\l,\g^\l,G_s^\sigma,\Gamma^\sigma$ & $\l$ and $\sigma$ are the
corresponding driving terms\\
$\g^\kappa,\Gamma^\kappa$ & traces with driving terms $\l(t)=\kappa\sqrt{1-t}$ and
$\sigma(s)\equiv \kappa$, resp.\\
$\Gamma_{u,v}$ & $G_u(\Gamma[u,v])$ \\
$\Gamma_u$ & $\Gamma_{u,\infty}$\\
$B^\R$ & reflection of $B$ about $\R$\\
\end{tabular}

\bigskip

\section{Self-similar curves}\label{s:self-similar}

We now describe the driving terms of curves $\gamma$ for which
$\widetilde \g_T$ and $\g$ are similar for each $T$.
Here we call two subsets $A,B\subset\H$ {\it similar} 
if they differ only by a dilation and translation fixing $\H$. 
We say that $\gamma$ is {\it self-similar} if  $\widetilde\g_T$ is similar to $\g$ for every $0 < T < 1$.

We then give an explicit construction of such curves. 

\begin{prop}\label{p:similar} The curve $\g$ is self-similar if
and only if $\l(t)=C + \kappa\sqrt{1-t}$ for some constants $C$ and
$\kappa$. Moreover, in this case, 
the renormalized curves $\g_T$ 
satisfy
\bes
\g_T-\g_T(0)=\g-\g(0)
\ees
for each $0<T<1$,
and the fixpoints of the map
$\g \mapsto \g_T$ are precisely the Loewner traces of $\l(t) = \kappa\sqrt{1-t}.$
\end{prop}

\begin{proof} Suppose $\widetilde \g_T= a(T) \g + b(T)$. Then by (\ref{renormlt}) and Section
\ref{s:definitions}, 
$\widetilde \g_T$ is driven by
\be\label{2drivers}
\l(T+t)=a\l(\frac{t}{a^2}) +b 
\ee
for $0 < t < 1-T$ and $0 < t < a^2$. Since these intervals must be the same, $a=\sqrt{1-T}$.
Setting $t=0$ in (\ref{2drivers}) we obtain $b=\l(T)-\l(0)\sqrt{1-T},$ and setting $t=1-T$ we obtain 
\bes
\l(T)=\l(1)+(\l(0)-\l(1))\sqrt{1-T}
\ees
as desired.
Conversely, if
$\l(t) = C + \kappa\sqrt{1-t}$, then $\l_T(t)= C/\sqrt{1-T}+\kappa\sqrt{1-t}$ by \eqref{renorml}
and so $\g_T$ is a translate of $\g$ for each $T$. 
Moreover $\g=\g_T$ if and only if $\lambda=\lambda_T$ if and only if $C=0$.
\end{proof}

Next we will construct curves $\gamma$ which are invariant under 
renormalization up to translation, hence obtaining the traces of 
$\kappa\sqrt{1-t}$ for some values of $\kappa.$ 
This approach has the advantage of being conceptual
and simple, but the disadvantage that it does not yield $\kappa.$ 
Each construction will be followed by an
explicit computation of the associated conformal maps, which then determines the associated constant
$\kappa$.

\subsection{Collisions}\label{ss:collisions} Fix $\theta$ with $0 < \theta < 1$. Let $D_\theta=\mathbb{H}\setminus
S_\theta$ where
$S_\theta$ is the line segment in
$\mathbb{H}$ from $0$ to $e^{i\pi \theta}$. See the upper right
corner of  Figure \ref{gt}. 
\begin{figure}[h]
\vskip 0.3truein
\centering
\psfrag{A}{$A$}
\psfrag{B}{$B$}
\psfrag{K}{$\kappa$}
\psfrag{zor}{$z/r$}
\psfrag{g}{$\gamma$}
\psfrag{reit}{$re^{i\pi\theta}$}
\psfrag{gt}{$g_t$}
\psfrag{gat}{$\gamma(t)$}
\psfrag{stz}{$\sqrt{1-t}~z$}
\psfrag{poma}{${}^{\pi(1-\theta)}$}
\psfrag{0p}{${}^{0^+}$}
\psfrag{0m}{${}^{0^-}$}
\psfrag{gt}{$g_t$}
\psfrag{At}{${}^{A\sqrt{1-t}}$}
\psfrag{Bt}{${}^{B\sqrt{1-t}}$}
\psfrag{Kt}{${}^{\kappa\sqrt{1-t}}$}
\psfrag{Dt}{$D_\theta$}
\psfrag{Rt}{$R_\theta$}
\psfrag{St}{$S_\theta$}
\psfrag{eit}{$e^{i\theta}$}
\psfrag{0}{$0$}
\psfrag{k}{$k$}
\psfrag{phi}{$G$}
\centerline{\includegraphics[height=4.25in]{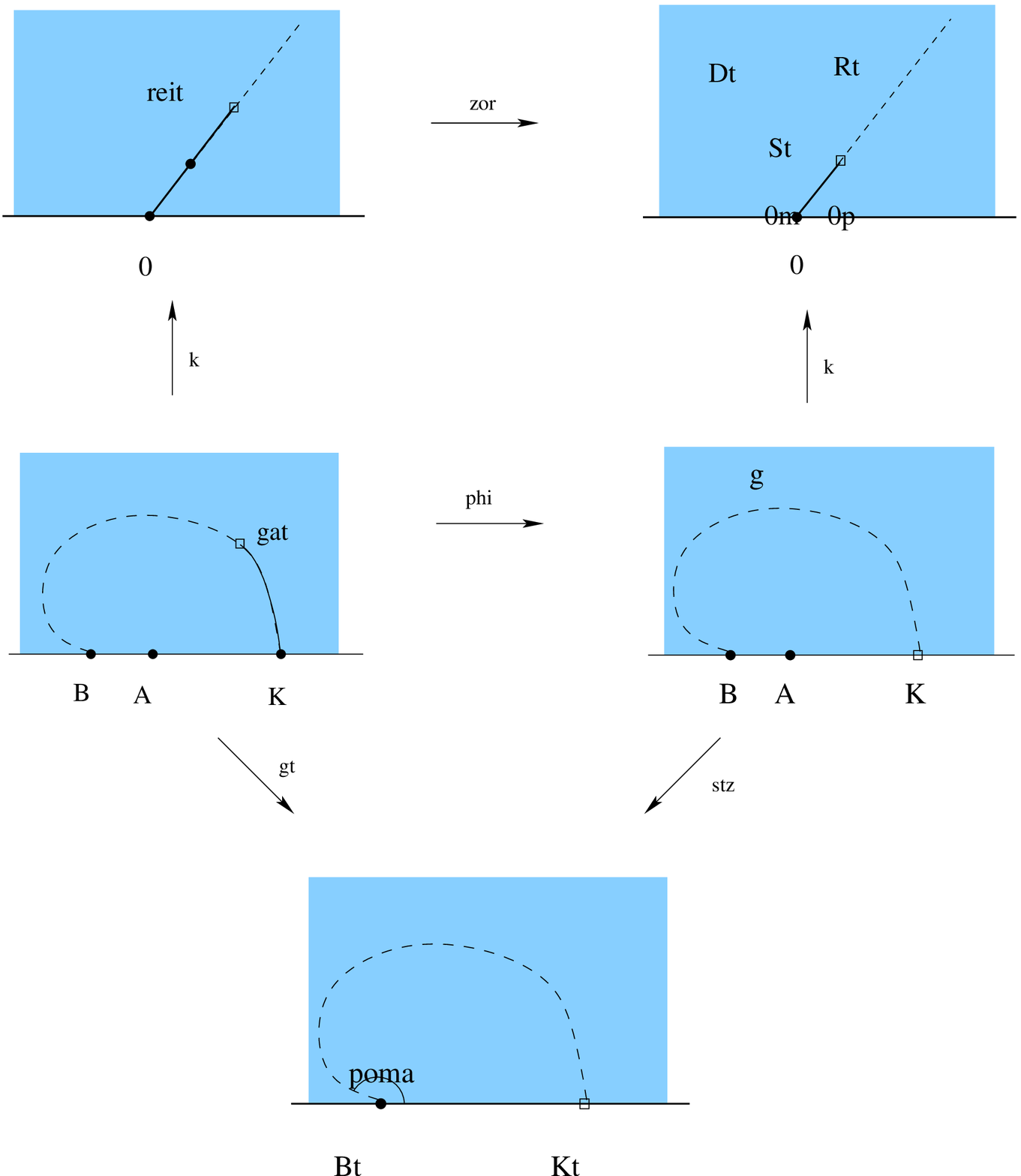}}
\caption{$\kappa > 4$:} 
\centerline{Loewner flow $z\mapsto z/r$ on the slit half-plane,}
\centerline{time changed 
Loewner flow $G$ on $\H$,} 
\centerline{and Loewner flow $g_t$ on $\H$.
}\label{gt}
\vskip 0.3truein
\end{figure}
Let $R_\theta$ be the ray $\{e^{i\pi \theta} r: r\geq1\}$ joining $e^{i\pi \theta}$ and $\infty$ in
$D_\theta.$ 
Viewed as a ``chordal Loewner trace'' in $D_\theta$ from $e^{i\pi \theta}$ to  $\infty$, 
$R_\theta$ has the
following similarity property: Parametrizing $R_\theta$ by $R_\theta(r)= e^{i\pi \theta} r, $ the conformal
map $z\mapsto z/r$ maps $D_\theta\setminus R_\theta[1,r]$ onto $D_\theta$ and maps
$R_\theta[r,\infty)$ onto $R_\theta.$
If we transplant the map $z\mapsto z/r$ to $\H$ by conjugating with a conformal map $k$ of
$\H$ onto $D_\theta$ then we will obtain 
a self-similar (in the sense of Proposition \ref{p:similar})
curve $\g=k^{-1}(R_\theta)$ provided $\infty$ is fixed. In other words, $k(\infty)$ 
must be fixed by the
map $z\mapsto z/r$. If $k(\infty)=\infty$ then $\g$ will be unbounded, and hence
have infinite half-plane capacity. 
The only other choices for the image of $\infty$ are
the two prime ends (boundary points) $0^+$ and $0^-$ of $D_\theta$ at $0.$
Choose $k$ so that $k(\infty)=0^+.$ 
Parametrize $\g=k^{-1}(R_\theta)$ by half-plane capacity, hcap, so that
$\g(0)=k^{-1}(e^{i\theta})$ and $\g(1)=k^{-1}(\infty)$ 
(we may replace $k(z)$ by $k(cz)$ for some constant $c>0$ so that
$\hcap(\g)=1$). Suppose $\gamma$ is driven by $\lambda$.
If $r>1$ is defined by 
$r e^{i\pi \theta}=k(\g(T)),$ then $G(z) = k^{-1}(\frac1r k(z))$ is a conformal map from 
$\H\setminus \g[0,T]$ to $\H$ fixing $\infty$  and hence must equal
$a(T)g_T+b(T)$ for some real constants $a$ and $b$. Since $G(\g[T,1]) = \g,$ 
$\widetilde \g_T$ is similar to $\g$ by (\ref{renormgt}).  Proposition
\ref{p:similar} then guarantees
$\l(t) = C + \kappa\sqrt{1-t}$. There is still one free (real) parameter in the definition of $k$, so
we may assume that $k^{-1}(e^{i\theta})=\kappa$ and thus
$\g(0)=\l(0)=\kappa$ and 
$\l(t) = \kappa\sqrt{1-t}$.
Notice that $\g$ ``collides'' with $\R$ at $\g(1)=k^{-1}(\infty)$ forming an angle of $\pi (1-\theta)$ 
with the half line $[k^{-1}(\infty),+\infty)$.

To compute the relation between $\theta$ and $\kappa$, we will compute the corresponding conformal maps
explicitly. 
The maps $k$ are the fundamental building blocks for the numerical conformal
mapping method called ``zipper'' \cite{MR2}.
By the Schwarz reflection principle or by Caratheodory's theorem, $k$
satisfies
\bes
\arg k(x)=
\begin{cases}
\pi \theta& {\rm for \ \ } A < x\\
\pi & {\rm for \ \ } B < x < A\\
0 & {\rm for \ \ } x < B.
\end{cases}
\ees
where $k(A)={0^-}$ and $k(B)=\infty$.
By Lindel\"of's maximum principle \cite[page 2]{GM},
\bes
\arg k(z)= \pi \theta + (1-\theta) \arg(z-A) - \arg(z-B) 
\ees
and so 
\be\label{kdef}
k(z)=ce^{i\pi \theta} \frac{(z-A)^{1-\theta}}{z-B}
\ee
where $c$ is a positive constant chosen so that the length of
$S_\theta$ will be equal to $1$. In fact for any choice $B<A$ and appropriate $c$, 
the right side of (\ref{kdef})
will be a one-to-one analytic map of $\H$ onto $D_\theta$ because it is the composition of $k$ with
a linear map.
Set
\be\label{gdef1}
G(z)=G_s(z)=k^{-1}\Bigl(\frac{1}{r} k(z)\Bigl),
\ee
where $r=r(s)$ will be determined shortly. Then
\bes
\dot G =  -{\frac{\dot r}{r^2}} \frac{k}{k'\circ G} = -\frac{\dot
r}{r}\frac{k}{k'}\circ G.
\ees
Computing $k'/k$ from (\ref{kdef}) and simplifying we obtain
\bes
\dot{G}=\frac{\dot
r}{r}\frac{(G-A)(G-B)}{(\theta G+(1-\theta)B-A)}.
\ees
Set $\dot{r}/r=\theta/2$, $AB=4$ and
$A+B=(A-(1-\theta)B)/\theta$.
Then
(\ref{LDE-G}) holds with constant $\sigma\equiv A+B$, and hence
(\ref{ODE}) holds with 
\bes
g_t(z)={\sqrt{1-t}}~{G_{s(t)}}\ees
and $\l(t)=(A+B)\sqrt{1-t}$. We can now compute the relation between
$\kappa=A+B$
and $\theta$:
If $A>0$ then since $AB=4$ and $A+B=(A-(1-\theta)B)/\theta$,
\be\label{ABdef} A=\frac{2}{\sqrt{1-\theta}}\quad\text{ and}\quad
B=2\sqrt{1-\theta}
\ee
and
\be\label{Ktheta}
\kappa=A+B=2\sqrt{1-\theta}+\frac{2}{\sqrt{1-\theta}}.
\ee
We also deduce that $r(s)=e^{s\theta/2}=(1-t)^{-\theta/2}$,
 and $c=2^{\theta}(1-\theta)^{\theta/2-1}$ since $k(\kappa)=e^{i\pi\theta}$. 
The trace $\gamma$ is a curve beginning at $\kappa$ which
``collides'' with $\R$ at $B=2\sqrt{1-\theta}$ forming an 
angle of $\pi(1-\theta)$ with $[B,\infty)$. Note that the interval
$0<\theta<1$ corresponds to the interval $4< \kappa <\infty$.
To obtain the maps for $-\infty < \kappa < -4$, simply reflect the
construction above about the imaginary axis.

In summary, we conclude:

\begin{prop}\label{p:slit}
Given $\kappa > 4$, set $\theta=2(1+\kappa/\sqrt{\kappa^2-16})^{-1}$ and
\bes
k(z)=e^{i\pi\theta}\frac{(z-2/\sqrt{1-\theta})^{1-\theta}}{z-2\sqrt{1-\theta}}
\ees
and
\bes
g_t(z)=(1-t)^{\frac{1}{2}}\,k^{-1}\Bigl((1-t)^{\frac{\theta}{2}}k(z)\Bigl)
\ees
Then $k$ is a conformal map of $\H$ onto $\H\setminus S_\theta$ where 
$S_\theta$ is a line segment in $\H$ beginning at $0$ and forming an
angle $\pi \theta$ with $[0,\infty)$, and $g_t$ satisfies the
Loewner equation
\bes
\dot{g_t}=\frac{2}{g_t-\kappa\sqrt{1-t}},
\ees
with $g_0(z)\equiv z$. The trace 
$\gamma=k^{-1}(\{re^{i\pi\theta}: r > 0\}\setminus S_{\theta})$
is a curve in $\H$ 
which meets $\R$ at angle $\frac{\pi}{2}$ at $\gamma(0)=\kappa$ and 
at angle $\pi(1-\theta$) at $\gamma(1)=2\sqrt{1-\theta}$. The case
$\kappa<-4$ can be obtained from the case $\kappa>4$ by reflecting about the
imaginary axis.
\end{prop}

In the statement of Proposition \ref{p:slit} we have replaced
$c$ (from (\ref{kdef})) by 1 for simplicity.
Indeed the definition of $\g$ and $G$ do not depend on the choice of
$c$. Changing $c$ only changes the size of the slit $S_\theta$.
Here the length of the slit is $|k(z_0)|$, where $z_0$ is the solution
of $k'(z_0)=0$.

\subsection{Spirals}\label{s:spirals}

Another type of region with a self-similarity property is a logarithmic spiral. 
Fix $\theta \in (0,\pi/2)$, set $\zeta=e^{i\theta}$  and consider the logarithmic spiral
\be\label{logspiral}
S_\theta(t) = e^{t\zeta} ,\quad -\infty\leq t\leq \infty.
\ee
Set $S^1=S_{\theta}[0,\infty)$, $D_\theta = \C\setminus S^1$ and let $R_\theta$ be the curve
$R_\theta(t)=S_\theta(-t), t\geq0.$ 
See  the upper right corner of Figure \ref{spiraldef}.
Viewed as a ``Loewner trace'' in $D_\theta$ from the boundary point $1$ of $D_\theta$ to the
interior point $0$,~ $R_\theta$ has the
following self-similarity property: 
$z\mapsto e^{t\zeta} z$ maps  
$D_\theta\setminus R_\theta[0,t]$ onto $D_\theta$ and
$R_\theta[t,\infty)$ onto $R_\theta.$
As before, it follows that any conformal map $k:\H \to D_\theta$ that fixes
$\infty$  sends $R_\theta$ to a curve $\g\subset\H$
driven by $\l(t) = a + \kappa\sqrt{\hcap(\g)-t}$. Notice that now the endpoint of 
$\g$ is an interior point of $\H,$ and because conformal maps are asymptotically
linear, we see that $\g$ is asymptotically similar to the logarithmic spiral
at the endpoint. 

To compute the relation between $\theta$ and $\kappa$, we will compute the corresponding conformal maps
explicitly as in the case $\kappa>4$. However, more work is required 
because Lindel\"of's maximum principle applies only to
bounded harmonic functions, which we do not have in this case. 
Let $\beta=k^{-1}(0)\in \H$
and let $\gamma_0=k^{-1}(-S_\theta)$. 
Then $\gamma$ is a Jordan arc in $\H$ from $k^{-1}(1)$ to $\beta$,
and $\gamma_0$ is an arc in $\H\setminus \gamma$ from $\beta$ to $\infty$.
See Figure \ref{spiraldef}.
\begin{figure}
\vskip 0.3truein
\centering
\psfrag{G}{$G$}
\psfrag{gt}{$g_t$}
\psfrag{tz}{$\sqrt{1-t}~z$}
\psfrag{zor}{$z\mapsto z/r$}
\psfrag{K}{$\kappa$}
\psfrag{Kt}{$\kappa\sqrt{1-t}$}
\psfrag{1}{$1$}
\psfrag{r}{$r$}
\psfrag{k}{$k$}
\psfrag{got}{$\gamma(t)$}
\psfrag{S1}{$S^1$}
\psfrag{mS}{$\color{blue}-S_{\theta}$}
\psfrag{SmS1}{$\color{red}R_{\theta}$}
\psfrag{log}{$\log z$}
\psfrag{exp}{$e^{z}$}
\psfrag{0}{$0$}
\psfrag{eit}{$e^{i\theta}$}
\psfrag{g}{\color{red}$\gamma$}
\psfrag{g0}{$\color{blue}\gamma_0$}
\psfrag{zmc}{$z-ce^{i\theta}$}
\psfrag{b}{$\beta$}
\psfrag{bt}{$\beta\sqrt{1-t}$}
\centerline{\includegraphics[width=4.25in]{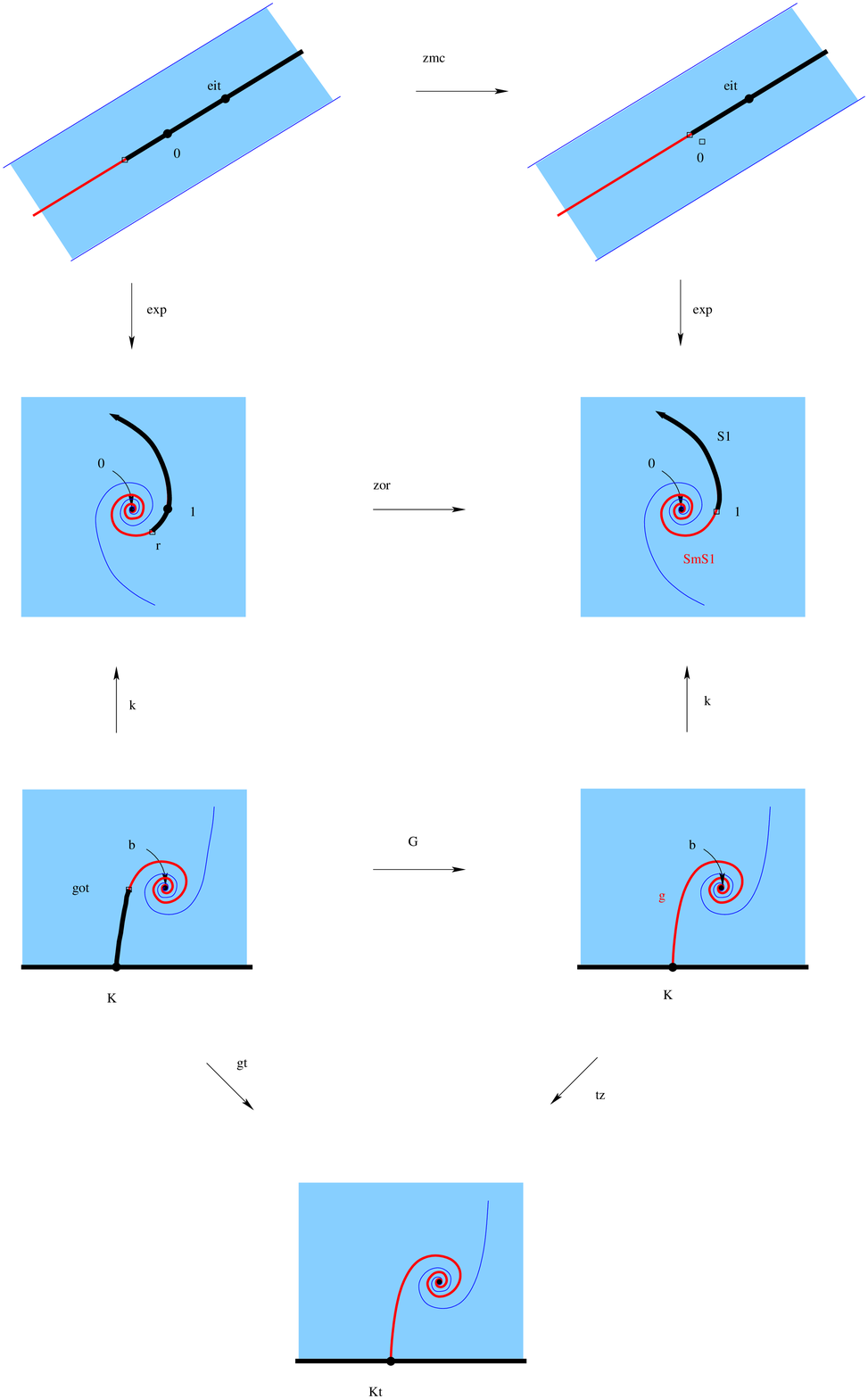}}
\caption{$0 < \kappa < 4$:} 
\centerline{Loewner flow $z\mapsto z/r$ on the complement of the
spiral,}
\centerline{time changed 
Loewner flow $G$ on $\H$,} 
\centerline{and Loewner flow $g_t$ on $\H$.
}\label{spiraldef}
\vskip 0.3truein
\end{figure}
We can define a single-valued branch of $\log k(z)$ in $\H\setminus\gamma_0$, 
with $\log k(k^{-1}(1))=0$,  so that for $z\in\R$ 
\begin{equation*}
\log k(z)\in e^{i\theta} \R^+  
\end{equation*}
and so that for $z\in \gamma=k^{-1}(R_\theta)$ 
\begin{equation*}
 \log k(z)\in -e^{i\theta}\R^+,  
\end{equation*}
where $\R^+=\{x>0\}$. 
Note that for $x \in \R$
\begin{equation*}
\overline{\log (x-{\beta})}=\log(x-{\overline{\beta}}) 
\end{equation*}
for continuous branches of the logarithms in $\overline{\H} \setminus
\gamma_0$, with $\lim_{x\to +\infty}\arg (x-\beta) =0$, and $\lim_{x\to
+\infty}\arg(x-\overline{\beta})=0$.
The function
\begin{equation*}
\log(z-{\beta})+e^{2i\theta}\log(z-{\overline{\beta}})=
e^{i\theta}[
e^{-i\theta}\log(z-{\beta})+e^{i\theta}\log(z-{\overline{\beta}})]
\end{equation*}
then maps $\R$ into the line with slope $\tan\theta$.
This suggests the following candidate for $k$:
\begin{equation}\label{k1def}
k_1(z)=(z-{\beta}){(z-{\overline{\beta}})^{e^{2i\theta}}},
\end{equation}
which is analytic in $\H$, with $k_1(\R)\subset S_\theta$, just like $k$.
Then
\begin{equation*}
\frac{k_1^\prime(x)}{k_1(x)}={e^{i\theta}}
\biggl(\frac{x(e^{i\theta}+e^{-i\theta})-
(\beta e^{i\theta}+\overline{\beta
e^{i\theta}})}{(x-\beta)(x-\overline{\beta})}\biggl),
\end{equation*}
which points in the direction $e^{i\theta}$ for $x>\kappa$ and in the direction
$-e^{i\theta}$ for $x< \kappa$, where $\kappa$ is the zero of $k_1'$. 
Thus as $x$ varies from $-\infty$ to
$+\infty$, ~$\log k_1$ traces a half line from $\infty$ to the tip
$\log k_1(\kappa)$
and then back again. 
Let $C$ be the boundary of a large half disk, given by the
line segment from $-R$ to $R$ followed by a semicircle in $\H$ 
from $R$ to $-R$. For large $|z|$, $k_1(z)$ is asympotic to
$k_2(z)=z^{1+e^{2i\theta}}$, and 
\begin{equation*}
\frac{\partial \arg k_2}{\partial \arg z} = 1+ \cos 2\theta > 0.
\end{equation*}
Thus $\arg k_2(z)$ increases as the semicircle is traced in the
positive sense
and the total change in $\arg k_2$ along the semi-circle
is $\pi(1+\cos 2\theta)$, which is at most $2\pi$. Thus as $z$ traces
the curve $C$, ~$k_1$
traces a subarc of $S_\theta$ from $k_1(-R)$ to $k_1(\kappa)$ and back  to
$k_1(R)$, followed by a curve, on which $|z|$ is large, from $k_1(R)$ to
$k_1(-R)$. 
By the argument
principle, $k_1$ is a conformal map of $\H$ onto $\C\setminus S^\kappa$
where $S^\kappa$ is the subarc of $S_\theta$ from $k_1(\kappa)$ to $\infty$. 

It follows directly from the definition of $S_\theta$ that if $\zeta_1,
\zeta_2\in S_\theta$ then $\zeta_1/\zeta_2\in S_\theta$, and so
\begin{equation*}
k(z)=\frac{k_1(z)}{k_1(\kappa)}.
\end{equation*}
is a conformal map of $\H$ onto $\C\setminus S^1$ such that
$|k(z)|\to \infty$ as $z\in \H \to \infty$.

Define
\begin{equation}\label{gdef2}
G(z)=G_s(z)= k^{-1}(\frac{1}{r}k(z))=k_1^{-1}(\frac{1}{r} k_1(z)),
\end{equation}
where  $r=r(s)$ will be determined shortly.
Then
\bes
\dot G =  -{\frac{\dot r}{r^2}} \frac{k_1}{k_1'\circ G} = -\frac{\dot
r}{r}\frac{k_1}{k_1'}\circ G.
\ees
Computing $k_1'/k_1$ from (\ref{k1def}) and simplifying we obtain
\bes
\dot{G}=-\frac{\dot r}{r}
\frac{(G-\beta)(G-\overline{\beta})}
{\Bigl((1+e^{2i\theta})G-(\overline{\beta}+\beta e^{2i\theta})\Bigl)}.
\ees
Set $\dot{r}/r=-(1+e^{2i\theta})/2$, with $r(0)=1$, $|\beta|=2$ and
$\beta+\overline{\beta}=(\overline{\beta}+\beta e^{2i\theta})/(1+e^{2i\theta})$.
Then
(\ref{LDE-G}) holds with constant $\sigma\equiv
\beta+\overline{\beta}=\kappa$, and hence
(\ref{ODE}) holds with 
\bes
g_t(z)={\sqrt{1-t}}~{G_{s(t)}}\ees
and $\l(t)=\kappa\sqrt{1-t}$. 
Note that $r(s)=e^{-s(\cos\theta)e^{i\theta}}\in R_\theta$ 
so that $z \mapsto rz$ maps  $S^1$ to
$S^1\cup R_\theta[0,t]$  for some $t>0.$
Thus for each $r>0$, 
$G$ is analytic on $\H\setminus\gamma[0,t]$ for some $t>0$ and 
maps $\H\setminus\gamma[0,t]$ onto $\H$.

We can now compute the relation between $\kappa$
and $\theta$: since $|\beta|=2$ and $(\overline{\beta}+\beta
e^{2i\theta})/(1+e^{2i\theta})=\beta+\overline{\beta}$, we conclude
\bes
\beta=2ie^{i\theta}
\ees
and
\bes
\kappa=-4\sin\theta.
\ees
The trace $\gamma$ is a curve beginning at $\kappa$ which
spirals around $\beta\in \H$.
Note that the interval
$0 < \theta <\frac{\pi}{2}$ corresponds to the interval
$-4< \kappa < 0$. To obtain $0 < \kappa < 4$, we just reflect the construction
about the imaginary axis; equivalently let $-\frac{\pi}{2}<\theta<0$.

In summary, we conclude:

\begin{prop}\label{p:spiral}
Given $0 < \kappa <  4$, set $\theta=-\sin^{-1}(\kappa/4)$,
$\beta=2ie^{i\theta}$ and
\bes
k(z)=\frac{(z-\beta)(z-\overline{\beta})^{e^{2i\theta}}}
{(\kappa-\beta)(\kappa-\overline{\beta})^{e^{2i\theta}}}
\ees
and
\bes
g_t(z)=(1-t)^{\frac{1}{2}}\,k^{-1}\Bigl((1-t)^{-\cos\theta e^{i\theta}}k(z)\Bigl).
\ees
Then $k$ is a conformal map of $\H$ onto $\C\setminus S^1$ where 
$S^1=\{e^{te^{i\theta}}: t\ge 0\}$ is a logarithmic spiral in $\C$ beginning at $1$
and tending to $\infty$
 and $g_t$ satisfies the
Loewner equation
\bes
\dot{g_t}=\frac{2}{g_t-\kappa\sqrt{1-t}},
\ees
with $g_0(z)\equiv z$. The trace 
$\gamma=k^{-1}(\{e^{-te^{i\theta}}: t > 0\})$
is a curve in $\H$ beginning at $\kappa\in \R$ and spiraling around
$\beta\in \H$. 
The case
$-4 < \kappa < 0$ can be obtained from the case $4>\kappa>0$ by reflecting about the
imaginary axis.
\end{prop}

\subsection{Tangential intersection} \label{s:tangential}

Now let $D_0$ be the domain $\H\setminus \{x+\pi i: x\leq 0\}$ and let $R_0$ be
the halfline $\{x+\pi i: x\geq 0\}$, see Figure \ref{tangdef}.
Let $k:\H \to D_0$ be a conformal map 
normalized by $k(\infty)=p_{-\infty},$ where $p_{-\infty}$ is the prime end 
$\lim_{x\to-\infty} x+\pi i/2.$ Then $\g=k^{-1}(R_0)$  has the self-similarity property:  translation $z\mapsto z-r$
maps $D_0\setminus[\pi i,\pi i+r]$ onto
$D_0$ fixing $p_{-\infty}$. 
In this case $\g$ intersects $\R$ at $t=\hcap(\g)$ tangentially.

Next we compute the conformal maps explicitly to show this case
corresponds to $\kappa=4$, and $\kappa=-4$
corresponds to the reflection of $D_0$ about the imaginary axis. These cases can also be obtained
as limits of the collision case as $\theta \to 0$ or $\theta\to 1$, or as limits of the spiral case
as $\theta\to -\frac{\pi}{2}$ or as $\theta\to \frac{\pi}{2}$.

\begin{figure}[h]
\vskip 0.3truein
\centering
\psfrag{G}{$G$}
\psfrag{k}{$k$}
\psfrag{g}{$\gamma$}
\psfrag{2}{$2$}
\psfrag{2t}{${}^{2\sqrt{1-t}}$}
\psfrag{4}{$4$}
\psfrag{4t}{${}^{4\sqrt{1-t}}$}
\psfrag{gt}{$g_t$}
\psfrag{sqrtz}{$\sqrt{1-t}~z$}
\psfrag{D0}{$D_0$}
\psfrag{0}{$0$}
\psfrag{pi}{$\pi i$}
\psfrag{pinf}{$p_{-\infty}$}
\psfrag{zpc}{$z-r$}
\centerline{\includegraphics[width=4.25in]{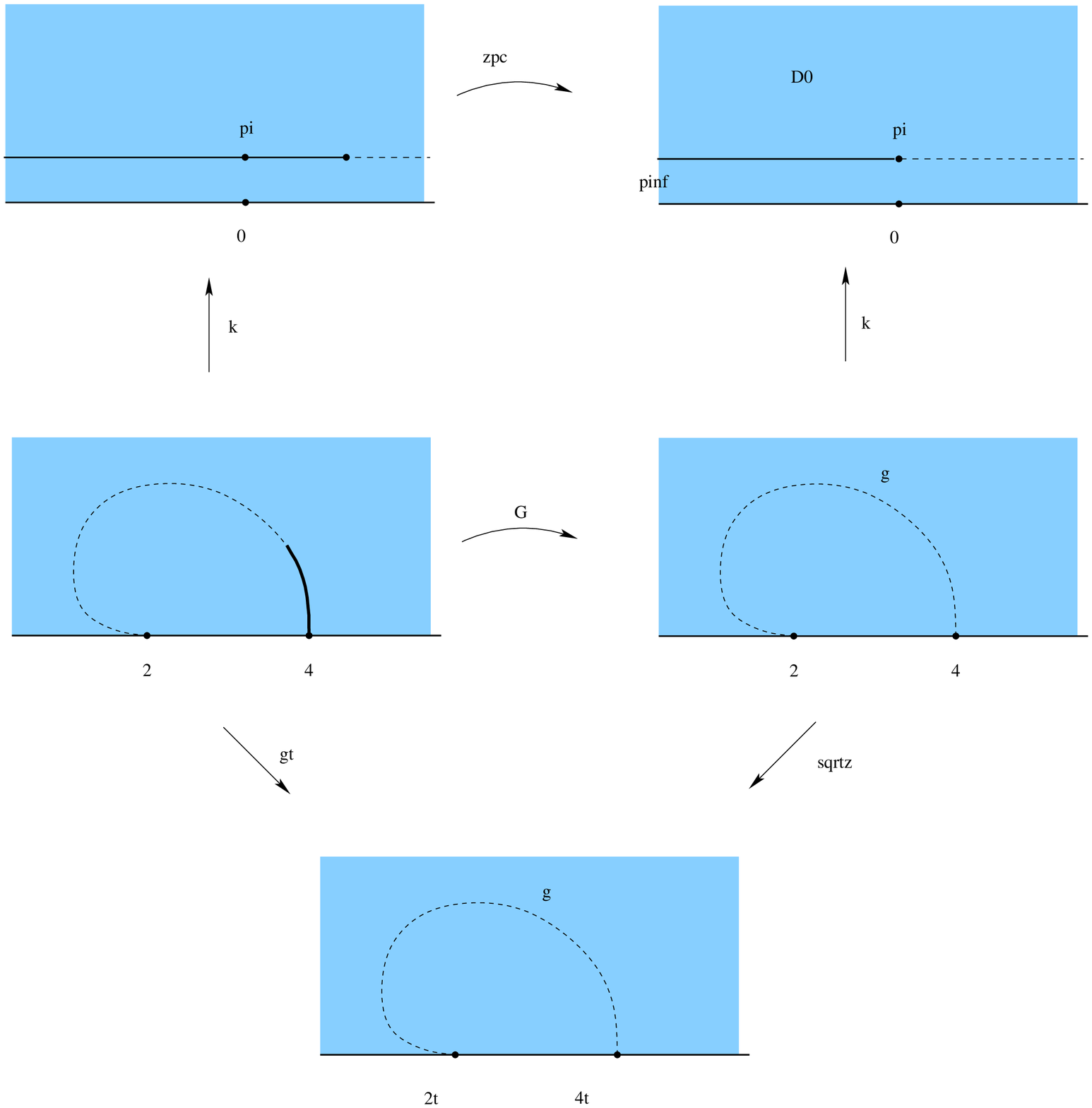}}
\caption{$ \kappa = 4$:} 
\centerline{Loewner flow $z\mapsto z-r$ on the slit half-plane,
}
\centerline{time changed 
Loewner flow $G$ on $\H$,} 
\centerline{and Loewner flow $g_t$ on $\H$.
}\label{tangdef}
\vskip 0.3truein
\end{figure}

As the case $|\kappa|<4$, we will construct the map $k:\H\to D_0$.
It is not enough to just construct an analytic function with the same
imaginary part, as was done with the logarithm in the case $\kappa >4$, since $k$ is
not bounded. Indeed, the identity function $z$ has  zero imaginary part on $\R$ yet is
nonconstant. It is perhaps easier to first construct a map $k_1$ of $\H$ onto $D_0$ which
maps $\infty$ to $\infty$. Then $k_1(z)-\log(z)$ will have no jump in the
imaginary part near $\infty$, and so it must behave like $cz+d$ for $z$
near $\infty,$ by the Schwarz reflection principle. Indeed the function defined by
\be\label{k1def2}
k_1(z)=z+1+\log(z)
\ee
is analytic in $\H$ and analytic across $\R\setminus \{0\}$. By calculus, $k_1$ is
increasing on $(-\infty,-1)$, decreasing on $(-1,0)$ and increasing on $(0,\infty)$
with $k_1(-1)=\pi i$.
The imaginary part of $k_1$ is zero on $(0,\infty)$ and equal to $\pi$ on $(-\infty,0)$.
Thus the image of $\R$ by $k_1$ is the boundary of $D_0$. Applying the argument principle
to regions of the form $\H\cap\{r < |z|<R\}$ for small $r$ and large $R$, we conclude that $k_1$ is
a conformal map of $\H$ onto $D_0$. Also $k_1(0)=p_{-\infty}$ so that 
$k_1(-1/z)$ maps $\H$ onto $D_0$ and sends $\infty$ to $p_{-\infty}$ as desired.
Set
\bes
k(z)=k_1(-1/(Az+B)),
\ees
where $A>0$ and $B\in \R$ and
\be\label{gdef3}
G(z)=k^{-1}(k(z)-r)
\ee
where $A$, $B$, and $r(s)$ will be determined shortly.
Then
\bes
\dot{G}=-\frac{\dot{r}}{k'\circ G}=\dot{r}\frac{(G+B/A)^2}{G+(B-1)/A}
\ees
and if $B=-1$, $A=1/2$, and $r=s/2$ then $\dot{r}=1/2$ and
(\ref{LDE-G}) holds with $\sigma=4$. Then the trace $\g=k^{-1}(R_0)$ is a curve in $\H$ from 
$\kappa=4$ to $2$ which is tangential to $\R$ at $2$. In summary, we conclude

\begin{prop}\label{p:tang}
Let
\bes
k(z)=\frac{4-z}{2-z}+\log\Bigl(\frac{2}{2-z}\Bigl)
\ees
and
\bes
g_t(z)=(1-t)^{\frac{1}{2}}\,k^{-1}\Bigl(k(z)+\frac{1}{2}\log(1-t)\Bigl).
\ees
Then $k$ is a conformal map of $\H$ onto $\H \setminus \{x+\pi i: x \le 0\}$ 
and $g_t$ satisfies the
Loewner equation
\bes
\dot{g_t}=\frac{2}{g_t-4\sqrt{1-t}},
\ees
with $g_0(z)\equiv z$. The trace 
$\gamma=k^{-1}(\{x+i: x > 0\})$
is a curve in $\H$ that begins at $4$, meeting $\R$ at right angles, and ending at $2$,
where it is tangential to $\R$.
The case
$\kappa=-4$ can be obtained from the case $\kappa=4$ by reflecting about the
imaginary axis.
\end{prop}

\subsection{Comments}

In  \cite{KNK}, an implicit equation for $g_t$ is found in each of the
cases above. They find the explicit conformal maps only in the special
case $\kappa=3\sqrt{2}$ (see Section 5 of \cite{KNK}). 
In this case  $\gamma=\gamma_\kappa$ is
a half circle and the example is closely related to the early work of
Kufarev \cite{K}.

The maps $g_t$ can also be computed without using Loewner's
differential equation by simply normalizing the maps we have
constructed at $\infty$.
For example, to determine $A$ and $B$ in the definition of $k$ in
(\ref{kdef}) when $\kappa> 4$, we want
\begin{equation*}
g_t(z)=\sqrt{1-t}~k^{-1}\bigl(\frac{1}{r}k(z)\bigl)=z + \frac{2t}{z}+\dots,
\end{equation*}
so that
\begin{equation*}
\frac{1}{r} k(z)=k\Bigl(\frac{g_t}{\sqrt{1-t}}\Bigl)
=k\circ\Bigl(\frac{z}{\sqrt{1-t}}+\frac{2t}{z\sqrt{1-t}}+\dots\Bigl).
\end{equation*}
Thus
\bes
\frac{(1-t)^{-\theta/2}}{r}\frac{\Bigl(1-\frac{B\sqrt{1-t}}{z}+\frac{2t}{z^2}+ O(\frac{1}{z^3})\Bigl)}{1-\frac{B}{z}}
=\biggl(\frac{1-\frac{A\sqrt{1-t}}{z}+\frac{2t}{z^2}+ O(\frac{1}{z^3})}{1-\frac{A}{z}}\biggl)^{1-\theta}.
\ees
Letting $z\to \infty$, we conclude that $r=(1-t)^{-\theta/2}$
and
\bes
1+\frac{B(1-\sqrt{1-t})}{z}+\frac{B^2(1-\sqrt{1-t})+2t}{z^2}+O(\frac{1}{z^3})=
\ees
\bes
1+\frac{(1-\theta)A(1-\sqrt{1-t})}{z}+\frac{1}{z^2}\Bigl((1-\theta)(A^2(1-\sqrt{1-t})+2t)+
\frac{1}{2}(1-\theta)(-\theta)A^2(1-\sqrt{1-t})^2\Bigl).
\ees
Equating coefficients, we obtain
\bes
B=(1-\theta)A \quad\text{ and }\quad A^2=\frac{4}{1-\theta},
\ees
which gives (\ref{ABdef}) as desired.

While it is possible to verify that $g_t$ satisfies Loewner's equation directly from its definition and avoid the use
of Section \ref{s:self-similar}, the former approach using the renormalization was what led us
to define $k$ in the first place. The renormalization idea is of critical
importance in Section \ref{s:collisions}.

\bigskip
If $k_1$ is the map (\ref{k1def2}) of $\H$ to the half-plane minus a horizontal half 
line as in
Section \ref{s:tangential}, then $-1/k_1(z)$ is a conformal map 
of the upper half
plane to the upper half-plane minus a slit along a tangential circle.
A careful analysis of the asympotics of the driving term $\lambda(t)$,
as $t\to 0$, for this curve was
made in  \cite{PV} using the Schwarz-Christoffel representation. With
the formula for the conformal map given here, an explicit expression
for the driving term can be given.

\section{Convergence of traces and driving terms}\label{s:close}

In general, it is not true that $||\l_n-\l||_{\infty}\to 0$ implies
$||\g_n-\g||_{\infty}\to 0$. 
A counterexample is described in page 116 of  \cite{La}.
All that can be concluded is Caratheodory convergence of $\H\setminus \g_n[0,t]$
to $\H\setminus \g[0,t]$. Neither does uniform convergence $\g_n\to\g$ imply $\l_n\to\l,$
see Figure \ref{notclose2} for a counterexample.
In Section \ref{s:uniform} we will give a
condition on a sequence of driving terms $\l_n$ 
that guarantees uniform convergence of the traces $\g_n$, and in
Section \ref{s:nearby} we will give a geometric condition
on traces that guarantees uniform convergence of their driving terms.

\subsection{Uniform convergence of traces}\label{s:uniform}

If $\l_n\to\l$ and if additionally the sequence $\{\g_n\}$ is known to be equicontinuous, then
uniform convergence follows easily. The following result makes use of this principle and
applies to a large class of driving terms. Let $\l_1,\l_2:[0,1]\to\R.$

\begin{thm}\label{t:uniform} For every $\eps>0$, $C<4$, and $D>0$ there is $\delta>0$
such that if 
\begin{equation}\label{unif1}
||\l_1-\l_2||_{\infty} < \delta
\end{equation}
and if 
\begin{equation}\label{unif2}
|\l_j(t)-\l_j(t')|\leq C |t-t'|^{1/2}
\end{equation}
whenever $|t-t'|<D$ and $j=1$ or $2$, then the traces $\g_1, \g_2$ are
Jordan arcs with
\begin{equation}\label{unif3}
\sup_{t\in[0,1]}|\g_1(t)-\g_2(t)|<\eps.
\end{equation}
\end{thm}

\begin{proof} Analogous to the definition of quasislit discs in
\cite{MR1} we call a domain of the form $\H\setminus\gamma$ a
$K-${\it quasislit half-plane} if
there is a $K-$quasiconformal map $F_\g:\H\to\H$ with $F_\g[0,i]=\gamma.$ Thus $\H\setminus\gamma$
is a quasislit half-plane (for some $K$) if and only if $\gamma$ is a quasiconformal arc in
$\overline\H$ that meets $\R$ non-tangentially.
Equivalently $\gamma\cup \gamma^\R$ is a quasiconformal arc, where
$\g^\R$ is the reflection of $\gamma$ about $\R$.
We will show first that $\H\setminus \g_1$ is a $K-$quasislit
half-plane with $K$ depending
on $C$ and $D$ only. 
Let
$$\xi_j(t) = \l_1(t+j D),\quad 0\leq j\leq \big\lfloor \frac1D \big\rfloor.$$
Then
$$g_t^{\l_1} = g_u^{\xi_N}\circ g_D^{\xi_{N-1}} \circ \cdots \circ
g_D^{\xi_1} \circ g_D^{\xi_0}$$
where 
$$N=\big\lfloor \frac{t}D \big\rfloor \quad{\text and} \quad u=t-N D.$$
By assumption, $|\xi_j(t)-\xi_j(t')|^{\frac{1}{2}}\le C|t-t'|$ for
$t,t'\in[0,D]$
and it follows from 
(\cite{L}, Theorem 2) that each ${g_D^{\xi_j}}^{-1}(\H)$ is a quasislit
half-plane (with $K=K(C)$).
It follows
that $\g_1[0,1]$ is the concatenation of $\big\lfloor \frac1D \big\rfloor +1$ 
$K-$quasislit half-planes with $K=K(C).$ For the sake of completeness, we sketch a proof of the 
fact that the concatenation $\alpha*\beta$ of two quasislits
$\alpha, \beta$ is a 
quasislit, see \cite{MR1} for the disc version. Let $h_{\alpha}:
\H\setminus\alpha\to\H$ be conformal
and let $F_{\beta}:\H\to\H$ be $K-$quasiconformal with
$F_{\beta}[0,i]=\beta$. Let 
$\psi(z)=\sqrt{z^2+1}$ be a normalized conformal map $\H\setminus[0,i]\to\H$.

\begin{figure}[h]
\psfrag{x}{$x$}
\psfrag{mx}{$-x$}
\psfrag{psi}{$\psi$}
\psfrag{phi}{$\phi$}
\psfrag{hinv}{$h_\alpha^{-1}$}
\psfrag{F}{$F$}
\psfrag{F2}{$F_{\beta}$}
\psfrag{G1}{$\alpha$}
\psfrag{G2}{$\beta$}
\psfrag{hinvG}{$h_\alpha^{-1}(\beta)$}
\psfrag{comp}{$F_{\alpha*\beta}$}
\centerline{\includegraphics[height=2.3in]{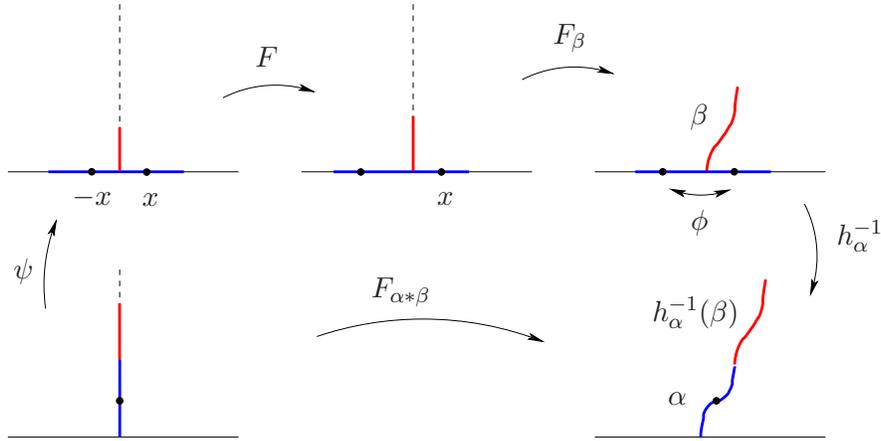}}
\caption{Concatenation $F_{\alpha*\beta}=h_{\alpha}^{-1}\circ F_{\beta}\circ F \circ \psi$.} 
\end{figure}

We would like to find a qc map $F:\H\to\H$ with $F(i \R_+) = i \R_+$ 
so that $h_\alpha^{-1}\circ F_{\beta}\circ F\circ\psi$ is qc on $\H$. Thus we need
$F_{\beta}(F(-x))=\phi(F_{\beta}(F(x)))$
for $x\in[-1,1],$ where $\phi: h_\alpha[\alpha]\to h_\alpha[\alpha]$ is the (decreasing) 
welding homeomorphism, defined through $y=\phi(x) \Leftrightarrow
h_\alpha^{-1}(x) = h_\alpha^{-1}(y).$
In order to construct such $F$, notice that
$\phi$ has a quasisymmetric extension to $\R$ by \cite{MR1} and \cite{L}.
It is easy to check that the function
$$F(x) =
\begin{cases}\ x \ \ \ \ \ &\text{if } x\geq0\\
\ F_{\beta}^{-1} (\phi(F_{\beta}(-x))) &\text{if } x<0
\end{cases} $$
is quasisymmetric. Extending $F$ to $\H$ in such a way that the imaginary axis is
fixed (this can be done by using the Jerison-Kenig extension \cite{JK}, (see also 
\cite{AIM}, Chapter 5.8)), we have obtained the desired map $F.$
Now $h_\alpha^{-1}\circ F_{\beta} \circ F \circ \psi(cz)$ is quasiconformal on
$\H\setminus[0,i]$, for an appropriate $c>0$, 
and continuous on $[0,i],$
hence quasiconformal on $\H.$ Thus $\alpha*\beta$ is a quasislit, and it follows by induction that
$\gamma_1$ is a quasislit. The same argument applies to $\gamma_2$.

Next, we claim that the parametrization of a $K-$quasislit by
half-plane capacity has modulus of continuity
depending on $K$ only. Denote $g_t:\H\setminus\gamma[0,t]\to\H$ the normalized map, then
$g_t(\gamma[t,t'])$ is a $K-$quasislit of capacity $t'-t,$ and hence of diameter
$\leq M\sqrt{|t'-t|}$
by \cite[Lemma 2.5]{MR1}. Because $g_t^{-1}$ is H\"older continuous with bound depending on $K$ only (by the John property 
of $\H\setminus\gamma$ and \cite[Corollary 5.3]{P}; see \cite{W} for the modifications to $\H$), 
the claim follows.

To finish the proof of the theorem, let $\lambda_{n,1}$ and $\lambda_{n,2}$ satisfy (\ref{unif2})
and $||\lambda_{n,1}-\lambda_{n,2}||_{\infty}\leq \frac1n.$ Passing to a subsequence we may assume
$\lambda_{n,j}\to\lambda$ uniformly. Denote $\gamma$ the Loewner trace of $\lambda.$ 
By the above equicontinuity of quasislits, we can pass to another
subsequence and may assume that there are curves $\gamma_j$ such that $\gamma_{n,j}\to \gamma_j$ uniformly.
Now the Theorem follows from the next lemma.
\end{proof}

\begin{lemma}
If $\lambda_n\to\lambda$ and $\gamma_n\to\gamma_\infty$ uniformly, then
$\gamma^\l=\gamma_\infty.$ That is, $\gamma_\infty$ is driven by $\lambda$.
\end{lemma}

\begin{proof} As $\gamma_n\to\gamma_\infty$, we have
$$\H\setminus\gamma_n[0,t] \to \H\setminus \gamma_\infty[0,t]$$
in the Caratheodory topology, for each $t.$ Hence
$$f_n(t,z)\to f_{\gamma_\infty}(t,z)$$
uniformly on compact subsets of $\H$, and so
$$f_n' \to f_{\gamma_\infty}' $$
locally uniformly. Using
$$\dot f_n = f_n' \frac{2}{\l_n-z}$$
it follows that
$$\dot f_n \to f_{\gamma_\infty}' \frac{2}{\l-z}$$
for each $z$ and each $t,$ and that $\dot f_n$ is uniformly bounded on $[0,T]\times \{z\}$
for each $T$ and $z.$ Thus
$$f_{\gamma_\infty}(t_1,z) - f_{\gamma_\infty}(t_2,z) = \lim_{n\to\infty} (f_n(t_1,z) - f_n(t_2,z)) = $$
$$ =\lim_{n\to\infty} \int_{t_1}^{t_2} \dot f_n(t,z) dt =  
\int_{t_1}^{t_2} f_{\gamma_\infty}^{\prime}(t,z) \frac{2}{\l(t)-z}dt$$
by dominated convergence. Hence
$$\dot f_{\gamma_\infty} = f'_{\gamma_\infty} \frac{2}{\l-z}$$
and the lemma follows.
\end{proof}

\subsection{Uniform convergence of driving terms}\label{s:nearby}

In this section we develop geometric criteria for two hulls to have
driving terms that are uniformly close.
Figure \ref{notclose2}
shows two hulls that are uniformly close but with large uniform distance
between their driving terms. 
\begin{figure}[h]
\psfrag{ep2i}{$\varepsilon + 2i$}
\psfrag{em2i}{$-\varepsilon + 2i$}
\psfrag{p1}{$p_1$}
\psfrag{p2}{$p_2$}
\centerline{\includegraphics[height=1.5in]{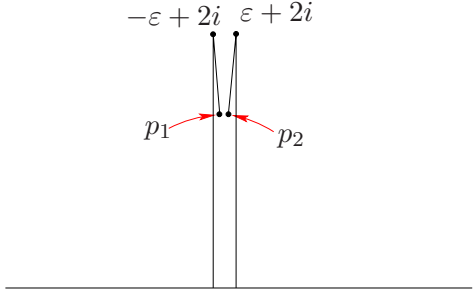}}
\caption{Close curves whose driving terms are not close.} \label{notclose2}
\end{figure}
If $R_\theta=\{re^{i\theta}: 0 <r<1\}$
and if $f(z)=\sqrt{z^2-4}$ then the image of $R_\eps$ by the map
$f(z)-\eps$ is similar to the left-hand curve in Figure
\ref{notclose2} for small $\eps$. 
The corresponding $\l$ is equal to $-\eps$ for $0\le t \le
1$ and is $>-\eps$ thereafter with a maximum value of approximately
$1$ by 
direct calculation (see Section \ref{ss:collisions}).  Likewise the image of $\R_{\pi-\eps}$ by the map
$f(z)+\eps$ is similar to the right-hand curve in Figure
\ref{notclose2}. The corresponding $\l$ is equal to $\eps$ for
$0 \le t \le 1$  and is $<\eps$ thereafter with a minimum value of
approximately $-1$ for small $\eps$. The corresponding hulls
$\gamma_j$ satisfy 
\bes
\sup|\gamma_1(t)-\gamma_2(t)| \le 2\eps,
\ees
but the driving terms satisfy
\bes
\sup |\l_1(t)-\l_2(t)|> 2-2\eps.
\ees

\bigskip

\begin{thm}\label{nearby}
Given $\eps>0$ and $c< \infty$,  suppose $A_1$, $A_2$ are hulls with 
$\diam A_j \le 1$  and such that there 
there exists a hull 
$B\supset A_1\cup A_2$  such that
\bes
\dist(\zeta,A_1) < \eps \hfil\text{ and }\hfil \dist(\zeta,A_2) <\eps
\hfil\text{ for all } \zeta\in B.
\ees
Suppose further that there are curves $\sigma_j\subset \H\setminus A_j$
connecting a point $p\in \H\setminus B$ to $p_j\in A_j$ with 
$\diam \sigma_j \le c\, \eps < \diam A_j$, for $j=1,2$.
If $g_j$ is the hydrodynamically normalized conformal map of
$\H\setminus A_j$ onto $\H$, for $j=1,2$,
then
\bes
|g_1(p_1)-g_2(p_2)|\le
2c_0\eps^\frac{1}{2}(c^{\frac{1}{2}}+\rho),
\ees
where $\rho$ 
is the hyperbolic distance from $p$ to $\infty$ 
in $\Omega=\C\setminus \widetilde B$, where
$\widetilde B=\overline{B\cup B^\R}\cup_j I_j$ and $I_j$ are the
bounded intervals in $\R\setminus\overline{B\cup B^\R}$.
\end{thm}

For example, if the Hausdorff distance between $A_1$ and $A_2$ is less than
$\eps$ and if $B$ is the complement of the unbounded component
of $\H \setminus \overline{A_1\cup A_2}$, then $\dist(\zeta,A_j) < \eps$ for all
$\zeta\in A_j$, $j=1,2$.  
The theorem also applies in some situations
where the Hausdorff distance between $A_1$ and $A_2$ is large.
If $p_1$ and $p_2$ are the tips of the curves in Figure \ref{notclose2} then
points $p$ which are close to $p_j$ have very large hyperbolic distance to $\infty$.
\begin{figure}[h]
\psfrag{A1}{$A_1$}
\psfrag{A2}{$A_2$}
\psfrag{B}{$B$}
\psfrag{s1}{$\sigma_1$}
\psfrag{s2}{$\sigma_2$}
\psfrag{p1}{$p_1$}
\psfrag{p2}{$p_2$}
\psfrag{p}{$p$}
\centerline{\includegraphics[height=1.5in]{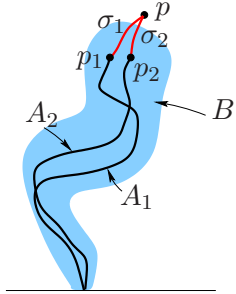}}
\caption{Close curves. } \label{hclose}
\end{figure}

The proof of Theorem \ref{nearby} will follow from several lemmas. 
The first proposition is well known, but we include it for the
convenience of the reader.

\begin{prop}\label{diam1} If $A$ is a hull, let $\widetilde A=
\overline{A\cup A^\R}\cup_j I_j$ where $A^\R$ is the reflection
of $A$ about $\R$ and $\{I_j\}$ are the bounded
intervals in $\R\setminus \overline{A\cup A^\R}$. If $g$ is the
hydrodynamically normalized conformal map of the simply connected domain
$\C^*\setminus \widetilde A$ onto $\C^*\setminus I$ 
where $I$ is an interval then
\be\label{capacity}
\diam{A} \le \diam{I} \le 4\diam{A}.
\ee
\end{prop} 
\begin{proof}
Let $G$ be the conformal map of $\Omega=\C^*\setminus\widetilde A$ onto
$\DD$ with
$G(z)=a/z + b/z^2+ \text{\rm O}(1/z^3)$ and $a > 0$. Then 
$g(z)=a(G + 1/G) +b/a$ so that 
\bes
I=[-2a+\frac{b}{a},2a+\frac{b}{a}]
\ees
and $|I|=4a$.
Since $1/G(z)$ is a conformal map of $\Omega$ onto $\C^*\setminus \DD$,
we conclude that $a=\Ca(\widetilde A)$, where 
$\Ca(E)$ denotes the logarithmic capacity of $E$.
If $E$ is a connected set, $\Ca(E)$ is decreased by projecting $E$
onto a line, and increased if $E$ is replaced by a ball containing
$E$. The capacity of an interval is one-quarter of its length and the
capacity of a ball is equal to its radius. 
Thus if $E$ is connected, its capacity is comparable to its
diameter and (\ref{capacity}) follows.
\end{proof}

The next
lemma will be used to bound $|g_j(p_j)-g_j(p)|$.

\begin{lemma} There exist $c_0 <\infty$ so that if $g$ is the
hydrodynamically normalized conformal map of a simply connected domain
$\H\setminus A$ onto $\H$ 
and if $S$ is a connected subset of $\H\setminus A$ then 
\be\label{diam}
\diam g(S) \le c_0 \max\bigl(\diam
S,(\diam{A})^{\frac{1}{2}}(\diam{S})^{\frac{1}{2}}\bigl).
\ee
In particular if $g$ is extended to be the conformal map of 
$\C^*\setminus \widetilde A$ onto $\C^*\setminus I$ 
where $I$ is an interval and ${\widetilde A}=\overline{A\cup A^\R}\cup_j I_j$,
where $A^\R$ is the reflection of $A$ about
$\R$ and $\{I_j\}$ are the bounded intervals in
$\R\setminus\overline{A\cup A^\R}$, then
\be\label{distance}
\dist(g(z),I)\le c_0
\max\bigl(\dist(z,A),\dist(z,A)^{\frac{1}{2}}\diam{A}^{\frac{1}{2}}\bigl),
\ee
\end{lemma}

\begin{proof}
We will prove (\ref{distance}), then use it to prove (\ref{diam}).
To prove (\ref{distance}) we may replace $A$, $I$, and $g(z)$ by $cA$,
$cI$, and $cg(z/c)$, so that without loss of generality $|I|=4$ and 
by (\ref{capacity}) $1\le  \diam A \le 4$. 
Fix $z=z_0\in \H$. If $\dist(z_0,A)\ge \diam A$ then
(\ref{distance}) follows from Koebe's estimate and the distortion
theorem. (See \cite{GM}, Corollary I.4.4 and Theorem I.4.5). Suppose
$\dist(z_0,A)<\diam A$ 
and let $\sigma$ be a straight line segment 
from $z_0$ to
$A$ with $|\sigma|=\dist(z_0, A)$. Set 
$\Omega=\C^*\setminus{\widetilde A}$,
set $\varphi(z)=|\sigma|/(z-z_0)$ and let $B=B(z_0,|\sigma|)$ be the ball centered at
$z_0$ with radius $|\sigma|$.  Then
\bes
\omega(\infty,\sigma,\Omega\setminus \sigma)\le \omega(\infty,
B,\Omega\setminus B)=\omega(0,\partial\DD,
\DD\setminus\varphi({\widetilde A})).
\ees
The circular projection of $\varphi({\widetilde A})$ onto $[0,1]$ 
is an interval $[|\sigma|/R,1]$ where $R\ge \frac{1}{2}\diam A$.
By the Beurling projection theorem \cite{GM}, Theorem III.9.2, and an
explicit computation, we obtain
\be\label{est1}
\omega(\infty,\sigma,\Omega\setminus\sigma) \le
\omega(0,\partial\DD,\DD\setminus [|\sigma|/R,1]) \le
\frac{4}{\pi}\tan^{-1}(\sqrt{\frac{|\sigma|}{R}})\le 
c_1\sqrt{|\sigma|}. 
\ee
Let $G$ be the conformal map of $\Omega$ onto $\DD$ with $G(\infty)=0$, 
with positive derivative at $\infty$. Then by
Beurling's projection theorem again,
\bes
\omega(\infty,\sigma,\Omega\setminus\sigma)=
\omega(0,G(\sigma),\DD\setminus G(\sigma))\ge
\omega(0,G(\sigma)^*,\DD\setminus G(\sigma)^*)
\ees
where $E^*$ is the circular projection of a set $E\subset \DD$ onto
$[0,1]$.  Again by an explicit computation
\be\label{est2}
\omega(0,G(\sigma)^*,\DD\setminus G(\sigma)^*)\ge \frac{1-r}{\pi}
\ee
where $r=\inf\{|z|: z\in G(\sigma)^*\}=\inf\{|z|: z\in G(\sigma)\}$.
By Koebe's $\frac{1}{4}$-theorem $r\ge r_0$, where $r_0$ does not depend on $A$.
As in Proposition \ref{diam1}, $g=(G+1/G)+b$, since $a=|I|/4=1$. 
Now if $w=G(z_0)$ and if $\zeta$ is the closest point in
$\partial \DD$ to $w$, set $x=\zeta+1/\zeta+b \in I$.  Then $|w-\zeta|\le 1-r$ and 
\be\label{est3}
|g(z_0)-x| = |w+1/w -(\zeta +1/\zeta)|=|w-\zeta||1-\frac{1}{w\zeta}|
\le 
(1-r)(1+\frac{1}{r_0}).
\ee
Thus by (\ref{est3}), (\ref{est2}), and (\ref{est1}),
\bes
\dist(g(z_0),I)\le c\sqrt{|\sigma|}=c\sqrt{\dist(z_0,A)}
\ees
proving (\ref{distance}).

To prove (\ref{diam}), if $\dist(S,\partial \Omega)\ge \diam S $ 
then for $z\in S$ by the Koebe distortion estimate and
(\ref{distance})
\bes
|g'(z)| \le 4\frac{\dist(g(z),I)}{\dist(z,A)} \le
c_0 \max(1,\bigl(\frac{\diam A}{\dist(z,A)}\bigl)^{\frac{1}{2}}).
\ees
If  $z_1, z_2\in S$, then by integrating $g'$ along the line segment from $z_1$ to $z_2$ (which is
contained in $\Omega$) we obtain
\bes
|g(z_1)-g(z_2)| \le c_0 \max(\diam S ,(\diam A\,\,  \diam S )^{\frac{1}{2}}.
\ees
If $\diam S \ge \dist(S,\partial\Omega)$ then we may rescale as in the
proof of (\ref{distance}) so that $|I|=4$ and $1\le \diam A\le 4$. 
Take $z_1\in S$ so that $\dist(z_1,\partial\Omega) \le \diam S$.  Then $S\subset
B=B(z_1,\diam S)$ and $B\cap \partial\Omega \ne \emptyset$.
As before, let $G:\Omega\to \DD$ with $G(\infty)=0$.
By (\ref{est1}) and (\ref{est2})
\bes
\frac{1-r}{\pi} \le \omega(\infty, B, \Omega\setminus B) \le c_1 \sqrt{\diam S },
\ees
where $r=\inf\{|z|: z\in G(B)\}$.
If $G(B)_*$ denotes the {\it radial} projection of $G(B)$ onto
$\partial \DD$, then by Hall's lemma \cite{DHp} and (\ref{est1})
\bes
|G(B)_*| \le 2\omega(0,G(B), \DD\setminus G(B)\le c_2
\sqrt{\diam S }.
\ees
\begin{figure}[h]
\psfrag{Gs}{$G(S)$}
\psfrag{Gshs}{$G(S)^*$}
\psfrag{Gss*}{$G(S)_*$}
\psfrag{r}{$r$}
\centerline{\includegraphics[height=2.3in]{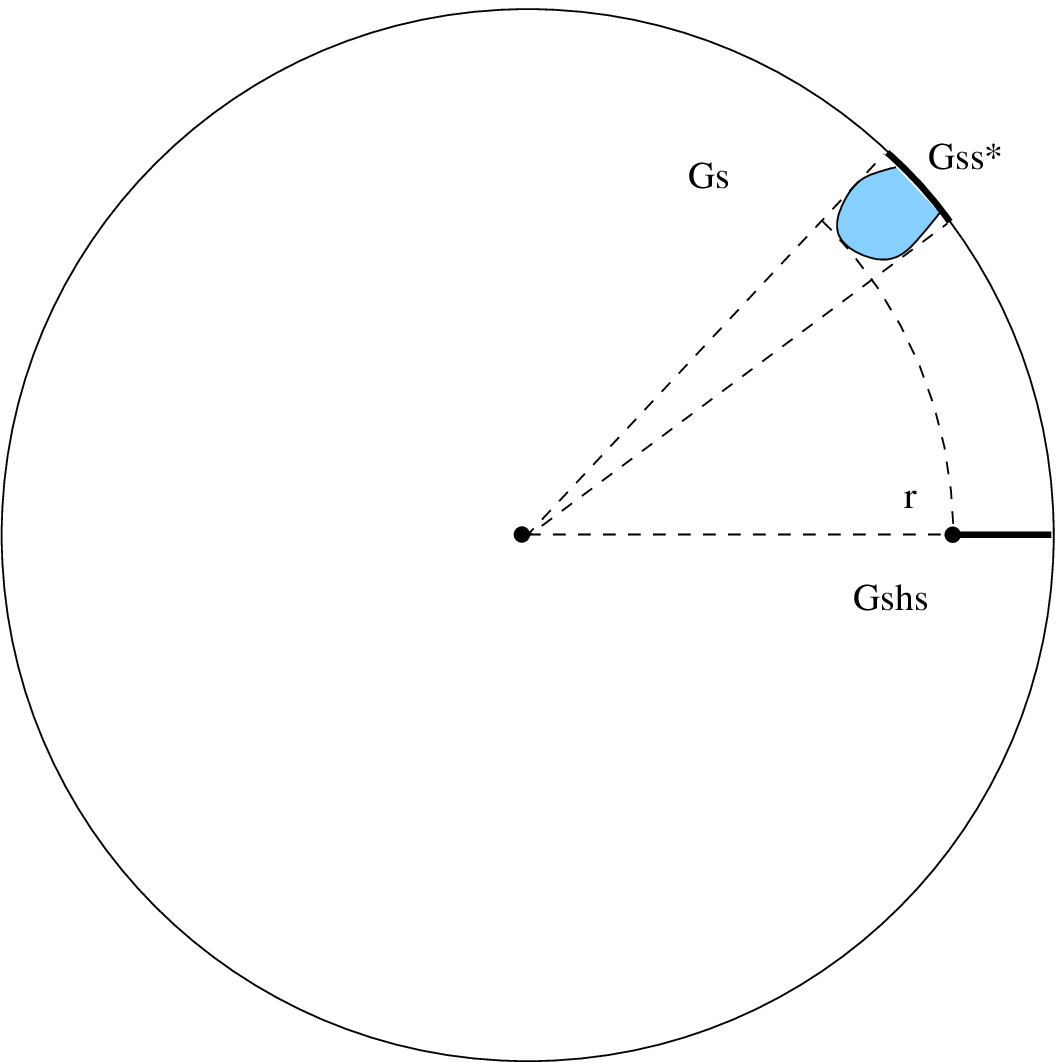}}
\caption{Diameter estimate via projections.}
\end{figure}

Since $G(S)$ is connected and $S\subset B$, we obtain
\bes
\diam G(S)\le c_3 \sqrt{\diam S }.
\ees
If $r>1/4$, this implies (\ref{diam}) . If $r<1/4$ then 
$\diam S  > \diam A \ge 1$ and by (\ref{distance})
\begin{eqnarray}
\diam g(S)&\le&2\sup_{z\in S} \dist(g(z),I) +|I|\\
&\le&c_0\sup_{z\in S}
(\dist(z,A), \dist(z,A)^{\frac{1}{2}})+4\le c_4 \diam S .
\end{eqnarray}
This proves (\ref{diam}).
\end{proof}

\begin{lemma}\label{calculation}
If $|z|>1$ then
\bes
\int_{-1}^1 \frac{dt}{|t-\frac{1}{2}(z+\frac{1}{z})|} =
2\log{\frac{|z|+1}{|z|-1}}=2\rho_{\C^*\setminus
\overline{\DD}}(z,\infty),
\ees
where $\rho_{\Omega}$ denotes the hyperbolic distance in $\Omega$.
\end{lemma}

\begin{proof}
The integral can be computed explicitly and then simplified using
$(z^\frac{1}{2} \pm z^{-\frac{1}{2}})^2= z +1/z \pm 2$.
\end{proof}

The next lemma follows immediately from Lemma \ref{calculation} and the
conformal invariance of the hyperbolic metric.

\begin{cor}\label{cauchy-rho}
If $I\subset \R$ is an interval and $\rho_{\C^*\setminus I}$ is the 
hyperbolic distance in $\C^*\setminus I$
then
\bes
\int_I \frac{dt}{|t-z|} = 2\rho_{\C^*\setminus I}(z,\infty).
\ees
\end{cor}

\bigskip

\begin{lemma}\label{AsubsetB}  Suppose $A\subset B$ are hulls such that
$\dist(\zeta,A) < \eps< 1$ for all $\zeta\in B$.
Let $g_A$ and $g_B$ be the hydrodynamically normalized conformal maps of
$\H\setminus A$ and $\H\setminus B$ onto $\H$ and let $\rho$ be the
hyperbolic distance from $z$ to $\infty$ in $\Omega=\C^*\setminus
\widetilde B$ 
where $\widetilde B=\overline{B\cup B^R}\cup_j I_j$ and $B^\R$ is the reflection
of $B$ about $\R$ and $\{I_j\}$ are the bounded intervals in
$\R\setminus \overline{B\cup B^\R}$. Then for $z\in \H\setminus B$ and $0 < \eps < \diam{A}$
\bes
|g_A(z)-g_B(z)|\le c_0 (\diam A)^{\frac{1}{2}}\rho\, \eps^{\frac{1}{2}}
\ees
where $c_0$ is the constant in (\ref{distance}). 
\end{lemma}

\begin{proof}
By (\ref{distance}), for $z\in \H\setminus A$
\bes
\Im g_A(z) \le c_0 (\diam A)^{\frac{1}{2}} \eps^{\frac{1}{2}}.
\ees
Let $I_B \subset \R$ denote the interval corresponding to $g_B(B)$ and set
$w=g_B(z)$, for $z\in \H\setminus B$.  
Then by (\ref{cauchy}) applied to $f=g_A\circ g_B^{-1}$ 
\bes
|g_A(z)-g_B(z)|=|g_A\circ g_B^{-1}(w)-w|=
\bigl|\frac{1}{\pi}\int_{I_B}\frac{\Im f(x)}{x-w}dx\bigl|
\le c_0 (\diam A)^\frac{1}{2} \eps^\frac{1}{2} 
\frac{1}{\pi}\int_{I_B} \frac{dx}{|x-w|}.
\ees
By Corollary \ref{cauchy-rho} 
and the conformal invariance of the hyperbolic metric
\bes
|g_A(z)-g_B(z)| \le  \frac{2c_0}{\pi} (\diam A)^{\frac{1}{2}}\rho\, \eps^{\frac{1}{2}}.
\ees
\end{proof}

\begin{proof}[Proof of Theorem \ref{nearby}]
Let $g_B$ be the hydrodynamically normalized conformal map of $\H \setminus
B$ onto $\H$. Then
\bes
|g_1(p_1)-g_2(p_2)|\le
|g_1(p_1)-g_1(p)|+|g_1(p)-g_B(p)|+|g_B(p)-g_2(p)|+|g_2(p)-g_2(p_2)|
\ees
The desired inequality for the first and last terms follows from 
(\ref{diam}) since $\diam \sigma_j \le c\eps<\diam A_j$.
The inequality for the second and third terms follows from Lemma \ref{AsubsetB}.
\end{proof}

In some circumstances it is preferable to use the hyperbolic metric in
$\Omega_1=\C^*\setminus \widetilde A_1$ instead of $\Omega=\C^*\setminus
\widetilde B$. The next lemma says that if $B$ is sufficiently close to $A_1$, then we
can do so. 

\begin{lemma}
If $\rho_{\Omega_1}(\infty,B) \ge c + \rho_{\Omega_1}(\infty,z)$ for
some $c>0$, then
\bes
\rho_{\Omega}(\infty,z)\le \rho_{\Omega_1}(\infty,z) +\log \frac{1}{1-e^{-c}}.
\ees
\end{lemma}

\begin{proof} Transfer the metric on $\Omega_1$ to the disk and use the explicit
form for the metric there.
\end{proof}

\bigskip

We would like to end this section by considering this question: if two curves are close together, were they  generated in approximately the same amount of time?
In other words, are their half-plane capacities close?
The lemma below addresses this.

\begin{lemma}\label{timesclose}
Suppose $A_1$, $A_2$ are hulls with
$\diam A_j \le 1$.
Suppose there exist a hull 
$B\supset A_1\cup A_2$  such that
\bes
\dist(\zeta,A_1) < \eps \hfil\text{ and }\hfil \dist(\zeta,A_2) <\eps
\hfil\text{ for all } \zeta\in B.
\ees
Then
$$|\hcap A_1 - \hcap A_2|\leq \frac{4}{\pi}c_0\, \eps^\frac{1}{2}.$$
\end{lemma}

\begin{proof}
Set $t_3=\hcap{B}-\hcap{A_1}>0.$
By (\ref{distance})
$|\Im g_{A_1}(z)|\le c_0\eps^\frac{1}{2}$ for $z\in B$.
By (\ref{2d}) applied to $f=g_{A_1}\circ g_B^{-1}$ we conclude
\bes
t_3 \le \frac{2}{\pi} c_0 \eps^\frac{1}{2}.
\ees
The same argument applies to 
$t_4=\hcap B - \hcap A_2$, and thus
\bes
|\hcap A_1-\hcap A_2|\le \frac{4}{\pi}c_0\, \eps^\frac{1}{2}.
\ees
\end{proof}

\section{The spiral}\label{s:spiral}

We have noticed in Proposition \ref{p:similar} that 
self-similar curves are driven by $\kappa\sqrt{1-t}.$  
We will first generalize this by proving that curves
which are ``asymptotically self-similar'' have driving terms
asymptotic to $\kappa\sqrt{1-t}.$ Then we will show that certain
spirals are asymptotically self-similar.

\subsection{Driving terms of asymptotically self-similar curves}

Let $\gamma^{(n)}:[0,1)\to\overline \H$ and $\gamma$ be Loewner traces
parametrized by half-plane capacity, with driving terms $\lambda^{(n)}$ and
$\l$.
We say that $\gamma^{(n)}$ converges to $\gamma$ in the Loewner topology and write
$\gamma^{(n)} \leadsto \gamma$
if for each $0<t<1$ we have
$$\sup_{0\leq \tau \leq t} |\lambda^{(n)}(\tau) -\lambda(\tau)| \to 0 \quad \text{as}\quad n\to\infty.$$

Fix $\kappa\in\R$ and let $\gamma^\kappa$ be the self-similar curve constructed
in Section \ref{s:self-similar},
driven by $\lambda^\kappa(t) = \kappa\sqrt{1-t}.$ We will first show that if
$\gamma:[0,1)\to\H$ is such that 
the renormalized curves $\gamma_T$, translated so as to start at
$\kappa,$ converge to  $\gamma^\kappa$ in the Loewner topology, then
$\lambda$ behaves like $\lambda^\kappa$ near $t=1.$ 

\begin{prop}\label{p:similar2} If $\gamma_T - \gamma_T(0)+\kappa
\leadsto \gamma^\kappa$ as $T\to 1,$ then
$\lambda$ has a continuous extension to $[0,1],$ and
$$ \lim_{t \rightarrow 1}
\frac{\lambda(t)-\lambda(1)}{\sqrt{1-t}}=\kappa.$$
\end{prop}

\begin{proof} 
Fix $a<1$ and set $T_n = 1-a^n.$ 
By (\ref{renorml}), $\gamma_T - \gamma_T(0)+\kappa$ is driven by
$$\phi_T(t) = \lambda_T(t) - \gamma_T(0)+\kappa =
\frac{\lambda(T+t(1-T)) - \lambda(T)}{\sqrt{1-T}}+\kappa.$$
By assumption, given $\eps>0$, there is an $n_0<\infty$ so that if
$n\ge n_0$ and $T_n\le t'\le T_{n+1}$ then
\bes
\biggl|\frac{\l(t')-\l(T_n)}{\sqrt{1-T_n}} +
\kappa-\kappa\sqrt{1-t}\biggl|< \eps,
\ees
where $t'=T_n+t(1-T_n)$ for some $0\le t\le 1-a$.
Thus
\be\label{tprime}
|\l(t')-\l(T_n)+\kappa(1-\sqrt{1-t})a^{\frac{n}{2}}| < \eps
a^{\frac{n}{2}}.
\ee
In particular if $t'=T_{n+1}$ then
\bes
|\l(T_{n+1})-\l(T_n)+\kappa (1-\sqrt{a})a^{\frac{n}{2}}|< \eps
a^{\frac{n}{2}},
\ees so that for $m > n \ge n_0$, by addition of these inequalities,
\bes
|\l(T_m)-\l(T_n)+ \kappa a^{\frac{n}{2}}(1-a^{\frac{m-n}{2}})| < \eps
a^{\frac{n}{2}}\frac{1-a^{\frac{m-n}{2}}}{1-\sqrt{a}}.
\ees
Since $a^k \to 0$, as $k\to \infty$, this proves $\{\l(T_m)\}$ is
Cauchy. Set $\l(1)=\lim \l(T_m)$. Then
\be\label{lambda1}
|\l(T_n)-\l(1)-\kappa a^{\frac{n}{2}}| < \eps
\frac{a^{\frac{n}{2}}}{1-\sqrt{a}}.
\ee
Adding (\ref{tprime}) and (\ref{lambda1}) gives
\bes
|\l(t')-\l(1)-\kappa\sqrt{1-t}a^{\frac{n}{2}}| <\eps
a^{\frac{n}{2}}\frac{2}{1-\sqrt{a}},
\ees
for $n \ge n_0$. Since $1-t'=(1-t)(1-T_n)=(1-t)a^n$ we conclude
\bes
\biggl|\frac{\l(t')-\l(1)}{\sqrt{1-t'}}-\kappa\biggl| <
\frac{2\eps}{(1-\sqrt{a})a^{\frac{1}{2}}},
\ees
for $T_n \le t'\le T_{n+1}$ and $n \ge n_0.$ The Proposition follows
by letting $\eps \to 0$.
\end{proof}

%

\subsection{Examples of asymptotically self-similar curves}\label{s:spiralexamples}

Next, we present a class of examples $\gamma$ that satisfy the assumption of 
Proposition \ref{p:similar2}.
In order to keep the proofs as short and simple as
possible, we will not give the most general definition, but 
restrict ourselves to the discussion of two specific examples.
However, in remarks during the proofs we will emphasize the assumptions that the
proofs really depend upon, allowing the reader to formulate and verify details of 
general conditions.

We first consider an infinite spiral that accumulates towards a given connected
compact set as in Figure \ref{star}. Consider the curve $\nu_0\in\DD$ given by
\begin{equation}\label{spiraldef2}
\nu_0(t) =   t e^{ \frac{i}{t-1}}, 0\leq t<1.
\end{equation}

\begin{remark}\label{r1}
{\rm
Denote $\hat\nu_0(t)$ the point on the ``previous turn'' with same argument as $\nu_0(t)$  
(in formula: $\hat\nu_0(t)=\nu_0(\hat t)$ where $\hat t = 1+1/(2\pi+1/(t-1))$). Notice that the domain
$\DD\setminus \nu_0[0,t]$, translated by $\hat\nu_0(t)$ and dilated by
$\pi i/(\nu_0(t)-\hat\nu_0(t)),$
converges to the slit half-plane $D_0= \H\setminus\{x+ \pi i \, :\, x \leq 0 \}$.
}\end{remark}

Let $A\subset\H$ be compact such that $\C\setminus A$ is
simply connected and
let $f:\DD\to \C\setminus A$ be a conformal map with $f(0)=\infty$.
Replacing $f(z)$ by $f(e^{i\theta_0}z)$ we may choose $t_0$ so that
$f(\nu_0(t_0))\in \R$ and $f(\nu_0(t))\in \H$ for $t>t_0$. 
Then $f(\nu_0(t)),$ $t_0\leq t<1$, parametrizes a curve that begins in
$\R$ and winds around $A$ infinitely often, accumulating at the outer 
boundary of $A.$ For example, Figure \ref{star} was created this way using the
numerical conformal mapping routine ``zipper'' \cite{MR1}.
We will show that this curve satisfies Theorem \ref{t:spiral} with
$\kappa=4$.
To this end, scale this curve so that its half-plane capacity is 1 (that is, consider the
curve $c f\circ\nu_0$ where $c^2 \hcap(f(\nu_0[t_0,1]))=1$), and reparametrize
by half-plane capacity. Call the resulting curve $\nu^A(t),$ and denote $\hat\nu^A(t)$
the point of the previous turn with same ``argument'' (formally, writing 
$\nu^A(t) = c f\circ\nu_0(u(t))$, we have $\hat\nu^A(t) = c
f\circ\nu_0(\widehat{u(t)})$),
where $\widehat{u(t)} $ is defined in Remark \ref{r1}).

\begin{figure}[h]
\psfrag{nuAt}{$\nu^A(t)$}
\psfrag{wnuAt}{$\hat \nu^A(t)$}
\psfrag{phit}{$\phi_t$}
\psfrag{psit}{$\psi_t$}
\psfrag{Lckm}{$k^{-1}$}
\psfrag{gt}{$L_t^{-1}\circ \displaystyle{\frac{g_t}{\sqrt{1-t}}}$}
\psfrag{0}{$0$}
\psfrag{4}{$4$}
\psfrag{2}{$2$}
\psfrag{R}{$R$}
\psfrag{B}{$B$}
\psfrag{mR}{$-R$}
\psfrag{Ht}{$H_t$}
\psfrag{pi}{$\pi$}
\psfrag{Ot}{$O_t$}
\psfrag{mzR}{$|z|=R$}
\psfrag{pi}{$\pi i$}
\psfrag{pinf}{$p_{-\infty}$}
\psfrag{wnu}{$\widetilde \nu$}
\psfrag{ptwnu}{$\psi_t(\widetilde \nu)$}
\centerline{\includegraphics[width=4.25in]{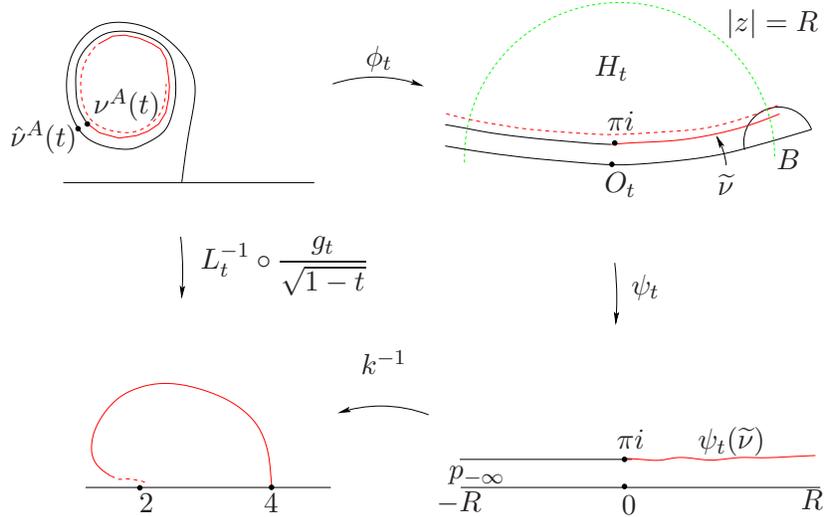}}
\caption{Decomposition of $g_t(z)$.}\label{spiralA}
\end{figure}

\begin{thm}\label{t:spiral driving} The curve $\nu^A$ satisfies the assumption
of Proposition \ref{p:similar2} with $\kappa=4,$ and consequently its driving term satisfies
$$ \lim_{t \rightarrow 1} \frac{\lambda(t)-\lambda(1)}{\sqrt{1-t}}=4.$$
\end{thm}

Recall the notation of Section \ref{s:tangential}, in particular the
slit half-plane 
$D_0,$ the conformal map $k:\H\to D_0$
and the curve $\gamma=k^{-1}(\{x+ \pi i \, :\, x \geq 0 \}),$ the trace of $4\sqrt{1-t}.$
The key feature of $\nu=\nu^A$ (and therefore the curve $\nu_0$
defined in (\ref{spiraldef2}))
is, roughly speaking, that $\H\setminus \nu[0,t]$ looks like $D_0$ when zooming in at $\nu(t).$
More precisely, we have

\begin{lemma}\label{l:caratheodory}

 For each $t\in[0,1),$ there is a linear map $\phi_t(z) = a_t z + b_t$
such that $\phi_t(\nu(t))=\pi i,$ and such that $\phi_t(\H\setminus \nu[0,t])$ converges to
$D_0$ in the Caratheodory topology (with respect to the point $1+\pi i$, say).
Furthermore, $O_t = \phi_t(\hat\nu(t)) \to 0$ as $t\to1.$
\end{lemma}

\begin{remark}\label{r2}
{\rm
Our curve $\nu_0$ is sufficiently smooth so that Caratheodory convergence will be enough.
For a general curve, we would need slightly stronger assumptions, see the remarks below.
}
\end{remark}

\begin{proof}
This is an easy consequence of the Koebe distortion theorem and Remark \ref{r1},
using that the distance $|\hat\nu_0(t)-\nu_0(t)|$ between consecutive turns
is asymptotic to $2\pi (1-t)^2$ and therefore much smaller than the distance
from $\nu_0(t)$ to $\partial\DD.$
\end{proof}

Next, let $\psi_t$ denote the conformal map from $H_t=\phi_t(\H\setminus \nu[0,t])$ onto
$D_0,$ normalized such that $\psi_t(\pi i)=\pi i,$  $\psi_t(O_t)=0$, and such that 
$\psi_t(\infty)$ equals the prime end $p_{-\infty}=k(\infty)$ (see Section \ref{s:tangential}).

\begin{lemma}\label{l:close to identity} For each $R>0$ and $\epsilon>0$ there is $t_0<1$ such that
\begin{equation}\label{cti}
|\psi_t(z) - z| < \epsilon \quad \text{for all}\quad z\in \widetilde \nu \cup B
\end{equation}
for $t>t_0$, where $\widetilde \nu$ is the component of  $\phi_t(\nu[t,1))\cap\{|z|<R\}$ containing $\pi i$,
and $B$ is the component of $H_t\cap \{|z-R|<2\pi\}$ containing $R+\pi i.$
\end{lemma}

\begin{proof}
A standard application of Caratheodory convergence, provided by Lemma \ref{l:caratheodory}, 
requires normalization of conformal maps at an interior point (such as $\pi i +1$).
To deal with our situation, denote $H_t=\phi_t(\H\setminus \nu[0,t])$ and
consider the conformal maps 
$\varphi_t:\DD\to H_t$
and $\varphi_0:\DD\to D_0,$ normalized by $\varphi_t(0)=\pi i+1$ and $\varphi_t'(0)>0.$
By Lemma \ref{l:caratheodory} we have $\varphi_t\to\varphi_0$ compactly as $t\to1.$
Denote $a_t,b_t$ and $c_t$ the preimages of $\infty, O_t$ and $\pi i$ under $\varphi_t,$
and denote $a,b,c$ the preimages of $p_{-\infty}, 0$ and $\pi i$ under $\varphi_0$.
Set $T_t=\varphi^{-1}_0\circ\psi_t\circ\varphi_t$ so that $T_t$ is the unique automorphism
of $\DD$ that maps $a_t,b_t,c_t$ to $a,b,c.$ 

It is not hard to see that $a_t\to a, b_t\to b$
and $c_t\to c$ as $t\to1:$ To prove $a_t\to a,$ fix $\rho>0$ large and
consider the vertical line segment 
$A=(-\rho,-\rho+\pi i)\subset D_0$ and notice that $A'=\varphi^{-1}_0(A)$ is a crosscut of $\DD$ of small diameter separating $0$ from $a.$ The extremal distance from $A$ to the boundary arc of $D_0$ 
between $0$ and $\pi i$ containing $\infty$ (that is, the image under $\varphi_0$ of the subarc of 
$\partial\DD$ between $b$ and $c$) is large (it is of the order $e^{\rho\pi}$). By conformal invariance,
the extremal distance between $\psi_t^{-1}(A)$ and the boundary arc between $O_t$ and $\pi i$
(that is, one turn of the spiral) is large, and it follows that the harmonic measure of
$\psi_t^{-1}(A)$ at $\pi i+1$ in $H_t$ is small. In particular,
there is $\rho'\leq\rho$ with $\rho'\to\infty$ as $\rho\to \infty$ ($\rho'=\rho/2$ will do)
such that for $t\geq t_0(\rho)$ 
the component $A_t$ of $(-\rho'+i\R) \cap H_t$ containing $-\rho'+i\pi/2$
separates $\pi i+1$ and $\psi_t^{-1}(A)$ in $H_t$. Hence $\varphi_t^{-1}(A_t)$ separates 
$0$ and $a_t$ in $\DD.$
Denote $\alpha=\varphi^{-1}_0(-\rho-1+i\pi/2)$ so that $\alpha$ is contained in the component of 
$\DD\setminus A'$ containing $a$. Because $\varphi_t(\alpha)\to\varphi_0(\alpha)$ as $t\to 1,$ 
$A_t$ separates $\pi i+1$ and $\varphi_t(\alpha)$ in $H_t$. Consequently, $\alpha$
is also contained in the component of $\DD\setminus \varphi^{-1}_t(A_t)$ containing $a_t$ and
we obtain $|a-a_t|\leq 2 (\diam A' + \diam \varphi^{-1}_t(A_t))$ which can be made arbitrarily
small by choosing $\rho$ large and $t\geq t_0(\rho).$
The convergence $b_t\to b$ and $c_t\to c$ can be proved in a similar fashion, replacing $A$ by 
small circular arcs centered at $0$ and $\pi i.$ We leave the details to the reader. 

It follows that $T_t\to\mathrm{ id}$ uniformly in $\DD.$ 
Hence uniform convergence to $0$ of $\psi_t(z)-z = \varphi_0(T_t(w))-\varphi_t(w)$, writing $w=\varphi^{-1}_t(z)$,
follows from the convergence $\varphi_t\to\varphi_0$ as long as $z$ stays boundedly close to $\pi i+1$ 
in the hyperbolic metrics of $H_t.$ This proves (\ref{cti}) on $\widetilde\nu\setminus \{|z-\pi i|<\delta\},$
for each $\delta>0.$ Because $\diam \varphi_t^{-1}(\{|z-\pi i|<\delta\}) < C\sqrt{\delta},$  
$c_t\in\varphi_t^{-1}(\{|z-\pi i|<\delta\}),$ $|c-T_t(c_t)|$ is small and $\varphi_0$ is continuous near $c,$
(\ref{cti}) also holds on $\widetilde\nu\cap \{|z-\pi i|<\delta\}$ by choosing $\delta$ small.
Finally, (\ref{cti}) on $B$ follows by extending $\psi_t^{-1}$ across the interval $[0,2R]$
using Schwarz reflection, and noticing that $\{|z-R|<2\pi\}$ is uniformly compactly contained in the extended 
domains $H_t$ for $t$ sufficiently large.
\end{proof}

\begin{remark}\label{r3}
{\rm 
For more general curves,
the validity of the conclusion of the previous lemma requires some mild regularity of $\nu$ in addition to the Caratheodory convergence of the rescaled domains $H_t$: Indeed, 
if $\phi_t(\nu(t))=\pi i$ cannot be joined to $\pi i+1$ within $H_t$
by a curve of diameter close to 1, then $\psi_t$ cannot be close to the identity near $\pi i.$
Assuming for instance that the component of 
$H_t\cap D(0,2\pi)$ is a John domain with $\pi i$ in its boundary
is enough to guarantee the conclusion of the lemma on $\widetilde\nu$.
Assuming that $\widetilde\nu$ is a $K(t)-$quasicircle with $K(t)\to1$ as $t\to1$ is enough to guarantee (\ref{cti}) on $B$.
}
\end{remark}

\begin{proof}[Proof of Theorem \ref{t:spiral driving}]
We need to show that the curves $\nu_T=g_T(\nu[T,1))/\sqrt{1-T},$  translated 
so as to start at $\kappa=4$, converge to $\gamma=\gamma^4$ in the Loewner topology as $T\to1.$
To see this, observe that $g_T/\sqrt{1-T}= L_T\circ k^{-1}\circ\psi_T\circ\phi_T$ for some linear
self-map $L_T(z)=\alpha_T z+\beta_T$ of $\H:$ Indeed, the map $k^{-1}\circ\psi_T\circ\phi_T$
is a conformal map from $\H\setminus\nu[0,T]$ onto $\H$ fixing $\infty.$

Next, we claim that $\alpha_T\to 1$ as $T\to1.$ Take $R$ large and consider 
the component $\widetilde \nu=\widetilde \nu(T,R)$ of  $\phi_T(\nu[T,1))\cap\{|z|<R\}$ containing $\pi i$. Assume $T$ is so large that $\phi_T(\nu[T,1))$ intersects $\{|z|=R\}$
(this is possible by Lemma \ref{l:caratheodory}).
Denote $e(T,R)$ the endpoint of $\widetilde\nu$ and let $T'$ be the corresponding time parameter, $\phi_T(\nu(T'))=e(T,R).$ If $T$ is large enough, then the line segment 
$S=S(T,R)=[e(T,R),\phi_T(\hat\nu(T'))]$
joining $e$ to the nearest point of the ``previous turn'' separates infinity from 
$\phi_T(\nu[T',1)).$ By the monotonicity of the half plane capacity, we obtain
$$\hcap  k^{-1}\circ\psi_T(\widetilde \nu) < 
  \hcap \big[ L_T^{-1}\circ g_T(\nu(T,1))/\sqrt{1-T}\big] < \hcap k^{-1}\circ\psi_T(\widetilde \nu\cup S).
$$
By Lemma \ref{l:close to identity} and the continuity of capacity (Lemma \ref{timesclose}) we see
(by letting $R\to\infty$ as $T\to1$) that
$$\hcap  k^{-1}\circ\psi_T(\widetilde \nu)\to 1$$ 
as $T\to1.$ By the subadditivity of $\hcap$ (\cite{La}, Proposition 3.42) and 
$\hcap  k^{-1}\circ\psi_T(S)\to0$, it follows that
$$\frac1{\alpha_T^2} =  \hcap \big[ L_T^{-1}\circ g_T(\nu(T,1))/\sqrt{1-T}\big] \to1.$$

Fix $t<1$. Then there is $R=R(t)$ (independent of $T$) such 
that $\phi_T(\nu[T,T+t(1-T)])\subset \widetilde\nu(T,R)$: Indeed, denote $R(T,t)$ the largest $R$
such that $\widetilde\nu(T,R) \subset \phi_T(\nu[T,T+t(1-T)])$, and assume to the contrary that there
is no upper bound on $R(T,t)$ as $T\to1.$ Then the argument of the previous paragraph shows that 
$$\limsup_{T\to1}\, \hcap  k^{-1}\circ\psi_T(\widetilde \nu(T,R(T,t)) = 1.$$ 
But 
$$\hcap  k^{-1}\circ\psi_T(\widetilde \nu(T,R(T,t))) = \hcap \big[ L_T^{-1}\circ g_T(\nu(T, T+t(1-T)))/\sqrt{1-T}\big] = 
\frac{t}{\alpha_T^2}$$
is bounded away from 1 as $T\to1,$ proving the existence of $R=R(t)$. (A direct estimate gives
that $R(t)$ is comparable to $s=\log1/(1-t)$). 

Thus Lemma \ref{l:close to identity} shows that $\psi_T$ is uniformly close to 
the identity on $\phi_T(\nu[T,T+t(1-T)])$, and it follows that $\nu_T(\tau)-\nu_T(0)+4$
is uniformly close to $\gamma^4(\tau)$, on $\tau\in[0,t].$ Now uniform convergence of 
the driving term of $\nu_T(\tau)-\nu_T(0)+4$ to the driving term of $\gamma^4$ is an easy
consequence of Theorem \ref{nearby}. Indeed, using the notation of
Theorem \ref{nearby}, fix $\tau \le t$ and  let $g_1$ and $g_2$ be the hydrodynamically
normalized conformal maps associated with the curves $\g^4[0,\tau]$
and $\nu_T([0,\tau])-\nu_T(0)+4$. So $g_1$ is equal to  $g_{\tau}$ from
Section \ref{s:tangential}, Tangential intersection.
Let $p_1=\g^4(\tau)$,
$p_2=\nu_T(\tau)-\nu_T(0)+4$, $p=\g^4((1-\eps)\tau+\eps)$ and let
$\sigma_j$ be the line segment from $p_j$ to $p$. The hyperbolic
distance from $\infty$ to $p$ is bounded independent of $\tau$ since
$\tau\le t<1$. 
\end{proof}

\begin{defn}\label{d:smallnorm}
We will say that a driving term $\mu:[0,1)\to\R$ has {\it local Lip 1/2 norm $\leq C$}
if there is $\delta>0$ such that
\begin{equation}\label{e:smallnorm}
|\mu(t)-\mu(t')|\leq C |t-t'|^{1/2} \quad\text{for all}\quad 0\leq t<t'<1 \quad
\text{with} \quad |t-t'|<\delta (1-t).
\end{equation}
We say that $\mu$ has {\it arbitrarily small local Lip 1/2 norm}, if for every $\epsilon>0$,
$\mu$ has local Lip 1/2 norm $\leq \epsilon.$
\end{defn}

\begin{prop}\label{p:sufficient}
If $\nu=\nu^A$ is the spiral constructed in Section \ref{s:spiralexamples}, then its driving term
$\l=\l^A$ has arbitrarily small local Lip 1/2 norm.
\end{prop}

\begin{proof} We need to show that the driving term of $g_t(\nu[t,t+\delta(1-t)))$  has small 
Lip 1/2 norm. Since scaling does not change the Lip 1/2 norm, this is equivalent to saying that
the renormalizations $\nu_t$, restricted to the interval $[0,\delta]\subset[0,1]$, have small
Lip 1/2 norm if $\delta$ is small. Using the analyticity of the basic spiral $\nu_0$ 
together with Koebe distortion, it is not hard to see that
$\nu_t[0,\delta]$ is a $K(\delta)$-quasislit half-plane with $K(\delta)\to1$ as $\delta\to0.$ Now the proposition 
follows from Theorem 2 in \cite{MR2}. 
\end{proof}

We end this section by noticing that the proofs of this section can be modified to
show the following: 

\begin{thm} If a sufficiently smooth (for instance asymptotically conformal)
Loewner trace $\g[0,1]$ has a self-intersection of angle $\pi(1-\theta)$ (see Figure \ref{gt})
with $\theta\in[0,1)$, then 
$$ \lim_{t \rightarrow 1}
\frac{\lambda(t)-\lambda(1)}{\sqrt{1-t}}=\kappa,$$
where
\be
\kappa=2\sqrt{1-\theta}+\frac{2}{\sqrt{1-\theta}}>4.
\ee
Similarly if $\g$ is asymptotically similar to the logarithmic spiral
$S_\theta$
(\ref{logspiral})
of Section \ref{s:spirals} then
$$ \lim_{t \rightarrow 1}
\frac{\lambda(t)-\lambda(1)}{\sqrt{1-t}}=\kappa,$$
where $\kappa=-4\sin \theta$.

\end{thm}

\section{Collisions}\label{s:collisions}

In this section we give sufficient conditions for the trace to
intersect itself in finite time.

\begin{thm}\label{t:collide} Suppose $\l(t)$ is continuous on $[0,1]$,
satisfies
\be\label{glim}
\lim_{t\to 1} \frac{\l(t)}{\sqrt{1-t}}=\kappa > 4,
\ee
and assume there is $C<4$ so that $\l$ has local Lip 1/2 norm less than $C$
(Definition \ref{d:smallnorm}).
Then the trace $\g[0,1]$ driven by $\t$ is a Jordan arc. 
Moreover, $\gamma_T(1)\in\R$ and
\be\label{colang}
\lim_{t\to 1}
\arg(\gamma_T(t)-\gamma_T(1))
=\pi\frac{1-\sqrt{1-16/\kappa^2}}{1+\sqrt{1-16/\kappa^2}},
\ee
provided $1-T$ is sufficiently small.
\end{thm}

Condition \ref{e:smallnorm} with $C<4$ is the smoothness condition referred to in
Theorem \ref{t:collision}. It will be used to prove that the trace
$\gamma_T$, for $T$ near 1, is
a curve which is close to the self-similar curve given in Proposition
\ref{p:slit}.
Recall that $\g_T=g_T(\g(T+t))/\sqrt{1-T}$, $t\in [0,1-T]$. 
The reason that $\g_T$
appears in the conclusion of Theorem
\ref{t:collide} instead of $\g$ is that the trace $\g$ might intersect
itself in $\H$ rather than in $\R$. Alternatively, we could have added the requirement
that $\Vert
\l(t)-\kappa\sqrt{1-t}\Vert_{\infty}$  be sufficiently small 
and then the conclusion holds with $\g_T$ replaced by $\g$,
as in the statement of Theorem \ref{t:collision}.

The method of proof also applies to the case $|\kappa|<4$ and yields the following
result:

\begin{thm}\label{t:sp} Suppose $\l(t)$ is continuous on $[0,1]$,
satisfies
\be
| \lim_{t\to 1} \frac{\l(t)}{\sqrt{1-t}}| < 4,
\ee
and assume there is $C<4$ so that $\l$ has local Lip 1/2 norm less than $C$.
Then the trace $\g$ driven by $\l$  is a Jordan arc.
Moreover, $\gamma$ is asymptotically similar to the logarithmic spiral at 
$\gamma(1)\in \H$.
\end{thm}

We first outline the idea underlying the proofs of 
Theorems  \ref{t:collide} and \ref{t:sp}, then give the 
details of the proof of Theorem \ref{t:collide}, and finally describe the adjustments
neccessary for the  proof of Theorem \ref{t:sp}.

\noindent{\it Outline} of the Proof of Theorems \ref{t:collide} and \ref{t:sp}. 
Since $\l$ has local Lip 1/2 norm less than 4, the trace $\g[0,t]$ is
a Jordan arc for each $t < 1$ by Theorem \ref{t:uniform}.  Let
$\Gamma(s)=\gamma(t(s))$ be the reparametrization of $\gamma$
described in Section \ref{s:scaling}.
Let $\Gamma^\kappa$ denote the self-similar curve driven by
$\s^\kappa(s)\equiv \kappa$ as in
Proposition \ref{p:slit}, and let $F^\kappa$ be the solution to
(\ref{LDE-F}) driven by $\s^\kappa$.
Fix $u_0$ large and
decompose $\Gamma$ as
$$\Gamma = \bigcup_{n=1}^{\infty} \Gamma_n$$
where $\Gamma_n= \Gamma_n(\s) = \Gamma[(n-1) u_0, n u_0].$
Then 
$$G_{(n-2) u_0}(\Gamma_n (\s)) = \Gamma_2(\s_{(n-2)u_0})$$
($\s_{(n-2)u_0}$ is $\s$ shifted by $(n-2)u_0$). By assumption, $\s_{(n-2)u_0}$ is close
to $\kappa$ if $n$ is large, hence $G_{(n-2) u_0}(\Gamma_n (\s))$ is close
to $\Gamma_2(\s^\kappa).$
Notice that $\Gamma_2(\s^\kappa)$ is close to a line segment if $u_0$ is large.
Now 
$$\Gamma_n = F_{(n-2)u_0} (G_{(n-2) u_0}(\Gamma_n (\s)))$$ so the Theorems follow from the fact that 
the map $F_{(n-2)u_0}$ is conformal (and contracting) in a
neighborhood of the fixpoint $B$ resp. $\beta$ of $F^\kappa$,
where the neighborhood does not depend on $n.$
In the case $\kappa>4,$ this is proved in Lemma \ref{interval}. If
$\kappa<4$, this is follows because  
$\beta\in\H$ and all $F_s$ are univalent in $\H.$

Now for the details.

\begin{proof}[Proof of Theorem \ref{t:collide}]
As before,
let $\Gamma^\kappa$ denote the self-similar curve driven by
$\s^\kappa(s)\equiv \kappa$ as in
Proposition \ref{p:slit}, and let $F^\kappa$ be the solution to
(\ref{LDE-F}) driven by $\s^\kappa$.
Set
$\theta=2(1+\kappa/\sqrt{\kappa^2-16})^{-1},$
\be\label{ABdef2}
A=\frac{2}{\sqrt{1-\theta}}\quad\text{ and }\quad B=2\sqrt{1-\theta},
\ee
so that $B < A < \kappa=A+B.$
Then by Proposition \ref{p:slit}, $\G^\kappa$ is a curve in $\H$ from
$\kappa$
to $B$, which meets $\R$ at angle
$\pi(1-\theta)=\pi\frac{1-\sqrt{1-16/\kappa^2}}{1+\sqrt{1-16/\kappa^2}}$.

Let
$\sigma(s)=e^{s/2}\lambda(1-e^{-s})$ be the (time changed) driving term
associated with $\lambda$ and let $\Gamma$ be the trace driven by $\s$.
Our first task is to prove that the solutions $F_s$ to the (time
changed) 
Loewner equation (\ref{LDE-F}) extend to be analytic in a fixed neighborhood of
$B$.

Define the interval
$I_s^\kappa=[x^\kappa_1(s),x^\kappa_2(s)]=G^\kappa_s(\G[0,s])$ as the
preimage of $\Gamma^\kappa[0,s]$,  by the map $F^\kappa_s$
so that $F^\kappa_s(x^\kappa_1(s))=F^\kappa_s(x^\kappa_2(s))=\kappa$
and $F^\kappa_s(\kappa)=\G(s)$.
By the Schwarz Reflection Principle,   $F^\kappa_s$ extends to be a
conformal map of $\mathbb{C}\setminus I_s^\kappa$ onto
$\mathbb{C}\setminus (\Gamma^\kappa[0,s]\cup \G^\kappa[0,s]^\R)$ where
$\Gamma^\kappa[0,s]^\R$ is the
reflection of $\Gamma^\kappa[0,s]$ about $\R$.  Note that by
(\ref{kdef}) and (\ref{gdef1}) we have that $F^\kappa_s(A)=A$. Since
$F^\kappa(x_j)=\kappa$ and $F^\kappa_s(\kappa)=\Gamma^\kappa(s)$ is
the tip of the slit $\Gamma^\kappa[0,s]$
we conclude
\begin{equation*}
0 < B < A < x^\kappa_1 < \kappa < x^\kappa_2 < \infty.
\end{equation*}
See Figure \ref{reflect}.
\begin{figure}[h]
\vskip 0.3truein
\centering
\psfrag{al}{$A$}
\psfrag{b}{$B$}
\psfrag{K}{$\kappa$}
\psfrag{FK}{$F^\kappa_s$}
\psfrag{x1}{$x^\kappa_1$}
\psfrag{x2}{$x^\kappa_2$}
\psfrag{Rgs}{$\Gamma^\kappa[0,s]^\R$}
\psfrag{gs}{$\G^\kappa[0,s]$}
\psfrag{Is}{$I^\kappa_s$}
\centerline{\includegraphics[width=4.25in]{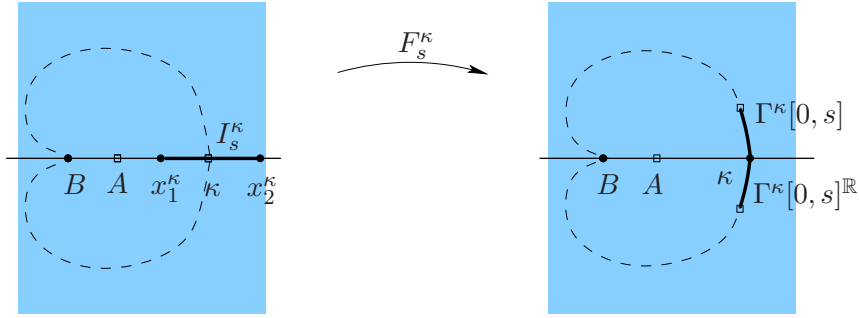}}
\caption{Extending the map $F^\kappa_s$ by reflection.
}\label{reflect}
\vskip 0.3truein
\end{figure}

Let $G_s=G_s^\s$ be the solution to (time changed) Loewner's differential equation 
(\ref{LDE-G}) 
driven by $\s$ and let $I_s=I_s^\s=[x_1(s),x_2(s)]=G_s([\Gamma[0,s])$
\bigskip

\begin{lemma}\label{interval} Suppose $\kappa > 4$.
Given $\delta >0$ there exists $\epsilon_1 > 0$ so that if
$||\sigma - \kappa ||_\infty < \epsilon_1 $ then
\begin{equation}
I_s \subset (A-\delta, \infty)
\end{equation}
where $A$ is defined  by (\ref{ABdef2}).
\end{lemma}

\begin{proof} Write $I_s=[x_1(s),x_2(s)]$. 
By (\ref{LDE-G})
\begin{equation*}
\dot{G}=\frac{2}{G-\sigma} +\frac{G}{2}=\frac{G^2-\sigma G
+4}{2(G-\sigma)}.
\end{equation*}
Thus $\dot{G} > 0$ whenever
\begin{equation}\label{Gint}
\frac{\sigma-\sqrt{\sigma^2-16}}{2} < G <
\frac{\sigma + \sqrt{\sigma^2-16}}{2} < \sigma.
\end{equation}
Recall from (\ref{ABdef}) and (\ref{Ktheta}) that $A+B=\kappa$ and
$AB=4$, so that $A$ and $B$ are roots of the equation
$\zeta^2-\kappa\zeta +4=0$. Since $B< A < \kappa$, 
we may suppose that $\delta$ is so small that $B+\delta<
A-\delta$. Then for $\epsilon_1$ sufficiently small and
$||\sigma-\kappa||_\infty<\epsilon_1$, we have that
\begin{equation}\label{intpos}
\frac{\sigma - \sqrt{\sigma^2-16}}{2} < B+ \delta < A -
\delta < \frac{\sigma+\sqrt{\sigma^2-16}}{2}.
\end{equation}
Thus $G_s(\kappa-\epsilon_1)$ is a continuous function of $s$
with $G_s(\kappa-\epsilon_1)< x_1(s)$ and
$G_0(\kappa-\epsilon_1)=\kappa-\epsilon_1 >
A-\delta$. Suppose there is an $s >0$ so that
\begin{equation*}
G_s(\kappa-\epsilon_1) < A-\delta.
\end{equation*}
Then we can find an $s_1 >0$ and $s_2>s_1$ so that
\begin{equation}\label{decr1}
G_s(\kappa-\epsilon_1)\ge A - \delta
\end{equation}
for $0 \le s \le s_1$ and
\begin{equation}\label{decr2}
B+\delta < G_s(\kappa-\epsilon_1) < A-\delta
\end{equation}
for $s_1 < s \le s_2$.
But by (\ref{Gint}) and (\ref{intpos}), $\dot{G} > 0 $ for $s_1 < s < s_2$. 
This contradicts (\ref{decr1}) and
(\ref{decr2}) and so $A -\delta< G_s(\kappa-\epsilon_1)< x_1(s)$. 
This completes the proof of Lemma (\ref{interval}).
\end{proof}

\noindent
{\bf Remark}.
There is no uniform upper bound on $x_2$. The expansion of $G_s$
about $\infty$  is given by
\begin{equation*}
G_s(z)=e^{s/2}z +\frac{2se^{s/2}}{z} + {\rm O}(\frac{1}{z^2}).
\end{equation*}
Thus
\begin{equation*}
x_2-x_1=|I_s|=4 C(I_s)=4e^{s/2}C(\G[0,s]\cup \G[0,s]^\R),
\end{equation*}
where $C(E)$ denotes the logarithmic capacity of $E$
and $R(\G[0,s])$ is the reflection of $\G[0,s]$ about $\R$.
Thus the length of $I_s$ is finite, but it tends to $\infty$
as $s\to \infty$. 

In particular each $F_s$ is analytic on the ball $\{z:|z-B|<\frac{A-B}{2})$ 
for $\epsilon_1$ sufficiently small.

\bigskip
To simplify the notation somewhat, we define
\be\label{pullbackdef}
\Gamma_{u,v}=G_u(\G[u,v])\quad\text{ and }\quad
\Gamma_{u,v}(s)=G_u(\G(u+s)),
\ee
for $0 \le s \le v-u$.  Then $\G_{u,v}(0)=G_u(\Gamma(u))=\s(u)$ 
and $\G_{u,v}(v-u)=G_u(\G(v))$.

\bigskip

\begin{lemma}\label{l:hyperbolic}
If (\ref{glim}) holds and if there is $C<4$ so that 
$\l$ has local Lip 1/2 norm less than $C$,
then given $\eps>0$ and $0 < u_0<\infty$, 
there is an $n_0<\infty$ so that
for $n\ge n_0$
\be\label{hyper}
\rho_{\H}(\Gamma_{(n-1)u_0,(n+1)u_0}(s),\Gamma^\kappa(s)) < \eps
\ee
for all $u_0 \le s \le 2u_0$,
where $\rho_{\H}$ is the hyperbolic distance in the upper half-plane $\H$.
\end{lemma}

\begin{proof} 

By (\ref{renorml}) and (\ref{glim}), $\l_T$ converges to
$\l^\kappa(t)=\kappa\sqrt{1-t}$ uniformly on $[0,1]$. Since the local Lip 1/2 norm is 
less than $C<4$, $\l_T$ satisfies the hypotheses of Theorem
\ref{t:uniform} on $[0,t_0]$ for each $t_0 < 1$. By Theorem 4.1 this
implies uniform convergence of $\g_T[0,t_0]$ to $\g[0,t_0]$ for each
$t_0 < 1$, as $T\to 1$.  
Since $u_0$ is fixed and $\Gamma$ is a reparametrization
of $\g$, the lemma follows. 
\end{proof}

\bigskip
\begin{lemma}\label{l:cones} 
For $u>0$, let $S_n$ be the line segment from $\Gamma(nu)$ to
$\Gamma((n+1)u)$.  Given $\epsilon>0$, there is $n_0< \infty $ and
$u<\infty$ so that for $n\ge n_0$  
\begin{equation}\label{segarg}
\bigl|\arg\bigl(\Gamma(nu)-\Gamma((n+1)u)\bigl)-\pi (1-\theta)\bigl| <
\epsilon,
\end{equation}
\begin{equation}\label{decay}
{\Im}{\Gamma((n+1)u)}\le \frac{1}{2} {\Im}{\Gamma(nu)},
\end{equation}
and
\begin{equation}\label{seghyp}
\rho_{\mathbb{H}}(\Gamma(s),S_n)\le \epsilon
\end{equation}
whenever $nu\le s \le (n+1)u$, where $\rho_{\mathbb{H}}$ is the hyperbolic
distance in $\mathbb{H}$.
\end{lemma}

Assuming Lemma \ref{l:cones} for the moment, we continue with the
proof of Theorem \ref{t:collide}.
Set
\begin{equation*}
C_{\epsilon}=\{z\in \mathbb{H}: |\arg{z}-\pi (1-a)|<\epsilon
\end{equation*}
and
\begin{equation*}
I_n=\mathbb{R}\cap[\Gamma(nu)-C_\epsilon].
\end{equation*}

\begin{figure}[h]
\vskip 0.3truein
\centering
\psfrag{piam}{$\pi \theta-\epsilon$}
\psfrag{piap}{$\pi(1-\theta)+\epsilon$}
\psfrag{gn}{$\Gamma(nu)$}
\psfrag{gnp}{$\Gamma((n+1)u)$}
\psfrag{Inp}{$I_{n+1}$}
\centerline{\includegraphics[height=1.5in]{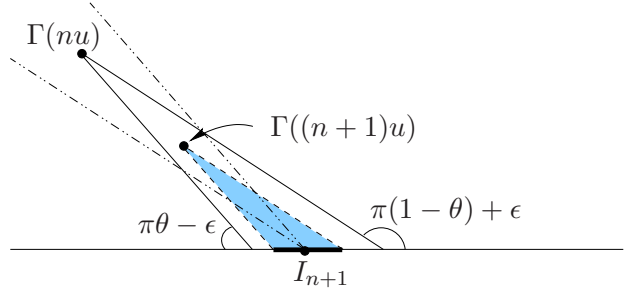}}
\nobreak
\caption{Cones with parallel sides
}\label{conefig}
\vskip 0.3truein
\end{figure}

By (\ref{segarg}),  $\Gamma((n+1)u)\in \Gamma(nu)-C_\epsilon$ and 
since the cones $\Gamma(nu)-C_\epsilon$ and
$\Gamma((n+1)u)-C_\epsilon$ have parallel sides, we conclude $I_{n+1}\subset
I_n$. See Figure \ref{conefig}.
By (\ref{decay}), ${\Im} \Gamma(nu) \to 0$ and hence $|I_n|\to 0$.
Set
\begin{equation*}
x_{\infty}=\bigcap I_n.
\end{equation*}
Note $x \in I_n$ if and only if $\Gamma(nu) \in x+C_\epsilon$. 
Thus $\Gamma(nu)\in x_{\infty} +C_{\epsilon}$ for all $n\ge n_0$.
By (\ref{seghyp}), 
\begin{equation*}
\{\Gamma(s): s >n_0 u\}\subset x_\infty+C_{M\epsilon},
\end{equation*}
where $M$ is a universal constant.
Letting $\epsilon \to \infty$, we obtain the Theorem.
\end{proof}

\vskip 0.5 truein

\begin{proof}[Proof of Lemma \ref{l:cones}]
To prove the lemma, we first verify that it holds for $\Gamma^\kappa$.
As before $F^\kappa_s$ is the inverse of $G_s=G_s^\kappa$.
By (\ref{kdef}) and (\ref{gdef1})
\bes
k\circ F^\kappa_s=e^{\theta s/2}k(z)
\ees
and hence
\bes
\frac{(F^\kappa_s(z)-A)^{1-\theta}}{F^\kappa_s(z)-B}=e^{\theta s/2}
\frac{(z-A)^{1-\theta}}{z-B}.
\ees
Since $F^\kappa_s(\kappa)=\Gamma^\kappa(s)$ we have that
\be\label{asym}
\Gamma^\kappa(s)-B=e^{-\theta s/2}
(\kappa-B)\biggl(\frac{\Gamma^\kappa(s)-A}{\kappa-A}\biggl)^{1-\theta}
\ee
Since $\kappa=A+B$, and $\G^\kappa(s) \to B$ we have 
\bes
\lim_{n\to
\infty}\frac{\G^\kappa(nu)-\G^\kappa((n+1)u)}{e^{-nu\theta/2}}=A(1-e^{-u\theta/2})
(1-A/B)^{1-\theta}.
\ees
Since $A>B$, (\ref{segarg}) holds for $n$ sufficiently large.
Also (\ref{seghyp}) follows from (\ref{asym}).
Choose $u$ so large that $e^{-\theta u/2}< \frac{1}{2}$ and then (\ref{decay})
follows from (\ref{asym}) with $s=nu$. 
We also note that by (\ref{asym}) 
\begin{equation}\label{Kangle}
|\arg(\Gamma^\kappa(s)-B) - \pi(1-\theta)| < \epsilon
\end{equation}
for $s\ge u$ if $u$ is sufficiently large.

To prove the lemma for $\Gamma$, given $\epsilon > 0$, 
by Lemma \ref{l:hyperbolic} we can choose $n_1$ so large that if $n\ge n_1$ and
$u \le s \le 2u$ then 
\begin{equation}\label{hyp}
\rho_{\mathbb H}(\Gamma_{(n-1)u,(n+1)u}(s),\Gamma^\kappa(s))<\epsilon^2
\end{equation}
and by (\ref{Kangle}) 
\begin{equation}{\label{garg}}
|\arg (\Gamma_{(n-1)u,(n+1)u}(s)-B) - \pi (1-\theta)| < 2\epsilon.
\end{equation}
Since $\Gamma^\kappa(s) \to B$ as $s\to \infty$
by (\ref{asym}),
we can also choose $u$ so large that
\begin{equation}{\label{small}}
|\Gamma_{(n-1)u,(n+1)u}(s)-B| < \epsilon,
\end{equation}
for $u \le s \le 2u$ and $n \ge n_1$, by Lemma \ref{l:hyperbolic}
again. 

Recall that by definition 
\begin{equation}{\label{gdef}}
\Gamma((n-1)u +s)=F_{(n-1)u}(\Gamma_{(n-1)u,(n+1)u}(s))
\end{equation}
for $0 \le s \le 2u$. 
Set
\begin{equation*}
h=F_{(n-1)u}(z+B) - F_{(n-1)u}(B). 
\end{equation*}
Then $h_1(z)=h((\frac{A-B}{2})z)$ is univalent on the unit disk
$\mathbb{D}$ by Lemma \ref{interval} and $h(0)=0$. By  (\ref{small}), (\ref{garg}), (\ref{gdef}),
and Theorem 3.5, \cite[page 95]{Du},
applied to $h_1/h_1^\prime(0)$, we conclude that
\begin{equation}\label{garg2}
|\arg (\Gamma(s) - F_{(n-1)u}(B))-\pi (1-\theta)| < 3\epsilon
\end{equation}
for $nu \le s \le (n+1)u$.
By (\ref{asym}) and (\ref{hyp}),
\begin{equation}{\label{dist}}
\biggl|\frac{\Gamma_{(n-1)u,(n+1)u}(2u)-B}
{\Gamma_{(n-1)u,(n+1)u}(u)-B}\biggl|<\frac{1}{2}.
\end{equation}
By the upper and lower estimates in the growth
theorem \cite[Theorem I.4.5]{GM}, 
\begin{equation}\label{growth}
\biggl|\frac{\Gamma((n-1)u+s_1)-F_{(n-1)u}(B)}
{\Gamma((n-1)u+s_2)-F_{(n-1)u}(B)}\biggl|
\le
(1+\epsilon)
\biggl|\frac{\Gamma_{(n-1)u,(n+1)u}(s_1)-B}
{\Gamma_{(n-1)u,(n+1)u}(s_2)-B}\biggl|,
\end{equation}
for $u \le s_1,s_2 \le 2u$.
By (\ref{garg2}),
(\ref{dist}), and (\ref{growth}) with $s_1=2u$ and $s_2=u$
we obtain (\ref{decay}) and then (\ref{segarg}) for $\Gamma$.
By (\ref{decay}), (\ref{garg2}), and (\ref{growth}) 
we have that (\ref{seghyp}) holds 
for $\Gamma$. 
\end{proof}

\vskip 0.5 truein

\begin{proof}[Proof of Theorem \ref{t:sp}]
As before, fix 
$u_0$ large and write
$$\g [0,1)= \bigcup_{n=1}^{\infty} \G_n(s)$$
where $\G_n= \G_n(\s) = \G[(n-1) u_0, n u_0].$
Since $\g$ is continuous on $[0,1)$ by the assumption $C<4$ and \cite{L}, 
we only need to show that $\diam \g[t,1)\to0$
as $t\to1$ in order to prove continuity of $\g$ on $[0,1].$
To do this, it suffices to show that $\diam \G_n$ decays exponentially.
Notice that
$$\G_n = F_{(n-2)u_0} (G_{(n-2) u_0}(\G_n (\s))),$$
and write $F_{(n-2)u_0}$ as a composition
$$F_{(n-2)u_0}=f_1\circ f_2\circ\cdots\circ f_{n-2},$$
where each $f_j$ corresponds to the driving term $\s$ restricted to $[(j-1) u_0, j u_0].$
By choosing $u_0$ large enough, we may assume that all $f_j$ except perhaps $f_1$ are
arbitrarily close to $F_{u_0}^\kappa$ (driven by the constant
$\s^\kappa\equiv \kappa$).
 
Writing $G=G^\kappa,$  (\ref{gdef2}) implies 
$$G_s'(\beta) = |\frac1r| = |e^{s(\cos\theta)e^{i\theta}}| >1$$
so that 
$$|{F_{u_0}^\kappa}'(\beta)| <1.$$
Choosing $u_0$ large,
Hurwitz' theorem implies that all $f_j$ ($j\geq2$) have a fixpoint $\beta_j$ near the fixpoint
$\beta$ of $F_{u_0}^\kappa$, and we may assume that the derivatives $f_j'(\beta_j)$ are arbitrarily
close to ${F_{u_0}^\kappa}'(\beta)$, hence uniformly bounded away from $1$ in absolute value.
As all $f_j$ are conformal maps of $\H,$ the Koebe distortion theorem (or normality) implies
the existence of a disc $D$ centered at $\beta$ and a constant $c<1$ such that
$f_j(D)\subset D$ and $|f_j'(z)|\leq c$ for all $z\in D$ and all $j\geq2$. Consequently,
$$\diam F_{(n-2)u_0}(D) \leq c^{n-2}.$$
Since for $u_0$ large enough, 
$$G_{(n-2) u_0}(\G_n(\s))=\Gamma_2(\sigma_{(n-2)u_0})\subset D$$
for all $n$, 
it follows that
$$\G_n = F_{(n-2)u_0}(\Gamma_2(\sigma_{(n-2)u_0})) \subset F_{(n-2)u_0}(D)$$ 
and the exponential decay of $\diam \Gamma_n$ follows at once.
By Theorem \ref{t:uniform}, $\Gamma_2(\sigma_{(n-2)u_0})$ converges to $\Gamma_2(\sigma^{\kappa})$
as $n\to\infty$. Because $\Gamma(\sigma^{\kappa})$ near $\beta$ is asymptotically similar to the
logarithmic spiral by Section \ref{s:spirals} and Koebe distortion (applied to $k^{-1}$ near $0$), 
$\Gamma_2(\sigma^{\kappa})$ 
rescaled (by a linear map) to have diameter $1$ converges to (a portion of) the logarithmic spiral
as $u_0\to\infty.$ Again by Koebe distortion (applied to $F_{(n-2)u_0}$) it follows that
$\Gamma_n$ rescaled to have diameter $1$ converges to the spiral, and the theorem follows.
\end{proof}

\vskip 0.5 truein
\begin{proof}[Proof of Theorem \ref{t:continuity}]
Let $\g$ be any of the spirals constructed in Section \ref{s:spiralexamples}
and let $\l$ be its driving term.
By Theorem \ref{t:spiral driving} and Proposition \ref{p:sufficient}, 
$r\l$
satisfies the assumptions of Theorems \ref{t:collide} and \ref{t:sp}
for $r>1$ and $r<1$ respectively, and Theorem \ref{t:continuity}
follows at once. 
\end{proof}

\end{doublespace}
\end{document}